\documentclass{dakota-article}
\pdfoutput=1

\usepackage{amsmath,amssymb,graphicx,bbm,xcolor}
\usepackage{amsthm,verbatim}
\usepackage{mathrsfs}
\usepackage{url}
\usepackage[linesnumbered,algoruled,boxed,lined]{algorithm2e}

\newcommand{\widebar}[1]{\bar{#1}}

\usepackage{caption}
\def\rotation{a}

\usepackage{enumitem}
\usepackage{mathtools}
\DeclarePairedDelimiter{\ceil}{\lceil}{\rceil}

\usepackage{colortbl}
\usepackage{booktabs}
\usepackage{tabularx}
\usepackage{tikz}
\usetikzlibrary{calc}
\pgfdeclarelayer{background}
\pgfdeclarelayer{foreground}
\pgfsetlayers{background,main,foreground}

\usepackage{mathdefs}

\newcommand{\bs}[1]{\boldsymbol{#1}}
\newcommand{\realsmap}[2]{\reals^{#1}\rightarrow\reals^{#2}}
\def\reals{\mathbb{R}}

\def\SNL{Center for Computing Research, Sandia  National Laboratories, Albuquerque, NM, USA.}
\def\UTAH{Department of Mathematics, and Scientific Computing and Imaging (SCI) Institute, University of Utah,  Salt Lake City, UT, USA.}

\def\mycorauthor{J.D Jakeman}
\corauthor{\mycorauthor}
\coremail{jdjakem@sandia.gov}
\title{Generation and application of multivariate polynomial quadrature rules}
\date{\today}
\author{J.D Jakeman\thanks{\SNL}, A.Narayan\thanks{\UTAH}}
\shortauthor{J.D Jakeman, A. Narayan}
\shorttitle{Multivariate polynomial quadrature rules}
\funding{See acknowledgements}

\begin{document}
\pagestyle{myheader}
\maketitle

\begin{abstract}
  The search for multivariate quadrature rules of minimal size with a specified polynomial accuracy has been the topic of many years of research. Finding such a rule allows accurate integration of moments, which play a central role in many aspects of scientific computing with complex models. The contribution of this paper is twofold. First, we provide novel mathematical analysis of the polynomial quadrature problem that provides a lower bound for the minimal possible number of nodes in a polynomial rule with specified accuracy. We give concrete but simplistic multivariate examples where a minimal quadrature rule can be designed that achieves this lower bound, along with situations that showcase when it is not possible to achieve this lower bound. Our second main contribution comes in the formulation of an algorithm that is able to efficiently generate multivariate quadrature rules with positive weights on non-tensorial domains. Our tests show success of this procedure in up to 20 dimensions. We test our method on applications to dimension reduction and chemical kinetics problems, including comparisons against popular alternatives such as sparse grids, Monte Carlo and quasi Monte Carlo sequences, and Stroud rules. The quadrature rules computed in this paper outperform these alternatives in almost all scenarios. 
\end{abstract}


\section{Introduction}\label{sec:intro}
Let $D \subset \R^d$ be a domain with nonempty interior. Given a finite, positive measure $\mu$ on $D$ and a $\mu$-measurable function $f: D \rightarrow \R$, our main goal is computation of
\begin{align*}
\int_D f(x) \dx{\mu(x)} 
\end{align*}
The need to evaluate such integrals arises in many areas, including finance, stochastic programming, robust design and uncertainty quantification. Typically these integrals are approximated numerically by quadrature rules of the form
\begin{align}\label{eq:quadrature}
  \int_D f(x) \dx{\mu(x)} &\approx \sum_{m=1}^M w_m f\left( x_m \right) 
\end{align}
Examples of common quadrature rules for high-dimensional integration are Monte Carlo and quasi Monte Carlo methods~\cite{Halton_NM_1960,Hammersley_H_book_1964,Niederreiter_book_1992,Sobol_S_IJMPC_1995}, and Smolyak integration rules~\cite{Gerstner_G_NA_1998,Gerstner_G_C_2003,Smolyak_SMD_1963}. Quadrature rules are typically designed and constructed in deference to some notion of approximation optimality. The particular approximation optimality that we seek in this paper is based on polynomial approximation.  

Our goal is to find a set of quadrature nodes $x_j \in D$, $j = 1, \ldots, M$, and a corresponding set of weights $w_j \in \R$ such that
\begin{align}\label{eq:quadrature-condition}
  \sum_{j=1}^M w_j p\left( x_j \right) &= \int_D p(x) \dx{\mu(x)}, & p &\in P_{\Lambda},
\end{align}
where $\Lambda \subset \N_0^d$ is a multi-index set of size $N$, and $P_\Lambda$ is an $N$-dimensional polynomial space defined by $\Lambda$. (We make this precise later.) $P_\Lambda$ can be a relatively ``standard" space, such as the space of all $d$-variate polynomials up to a given finite degree, or more intricate spaces such as those defined by $\ell^p$ balls in index space, or hyperbolic cross spaces.

In this paper we will present a method for numerically generating polynomial based cubature rules.\footnote{ Although the focus of this paper is on polynomial based quadrature rules, and the review of quadrature rules contained here reflects this focus, other types of quadrature rules exist. Polynomial based cubature rules rules are useful when the integrand $f$ has high-regularity. However when the function has less regularity, such as piecewise continuity alternative rules based upon other basis functions, such as piecewise polynomials. For example Simpson's rule for univariate functions, sparse grids based upon piecewise polynomials for multivariate functions~\cite{Bungartz_G_AN_2004} and the numerous adaptive versions of these methods e.g.~\cite{Pfluger_PB_JC_2010}} This paper provides two major contributions to the existing literature. Firstly we provide a lower bound on the number points that make up a polynomial quadrature rule. Our analysis is straightforward, but to the authors knowledge this is the first reported bound of its kind. Our second contribution is a numerical method for generating quadrature rules that are exact for a set of arbitrary polynomial moments. Our method has the following features: 
\begin{itemize}
  \item Positive quadrature rules are generated.
  \item Any measure $\mu$ for which moments are computable are applicable. Many existing quadrature methods only apply to tensor-product measures. Our method constructs quadrature rules for measures with, for example, non-linear correlation between variables.
\item Analytical or sample-based moments may be used. In some settings it may be possible to compute moments of a measure exactly, but in other settings only samples from the measure are available. For example, one may wish to integrate a function using Markov Chain Monte Carlo-generated samples from a posterior of a Bayesian inference problem.
\item A quadrature that is faithful to arbitrary sets of moments may be generated. Many quadrature methods are exact for certain polynomial spaces, for example total-degree or sparse grid spaces. However, some functions may be more accurately represented by alternative polynomial spaces, such as hyperbolic cross spaces. In these situations it may be more prudent to construct rules that can match a customized set of moments. 
\item Efficient integration of ridge functions is possible. Some high-dimensional functions can be represented by a small number of linear combinations of the input variables. In this case it is more efficient to integrate these functions over this lower-dimensional coordinate space. Such a dimension-reducing transformation typically induces a new measure on a non-tensorial space of lower-dimensional variables. For example, a high-dimensional uniform probability measure on a hypercube may be transformed into a non-uniform measure on a zonotope (a multivariate polygon). 
\end{itemize}
Our algorithm falls into the class of moment-matching methods. There have been some recent attempts at generating quadrature using moment matching via optimization apporaches. These methods frequently either start with a small candidate set and add points until moments are matched~\cite{Mehrotra_P_SJO_2013}, or start with a large set of candidate points and reduce them until no more points can be removed without numerically violating the moment conditions~\cite{ryu_extensions_2014,vahid2017,Vandebos_KD_JCP_2017}. These approaches sometimes go by other names, such as scenario generation or scenario reduction methods.

This paper presents a quadrature/scenario reduction moment matching method based upon the work in ~\cite{ryu_extensions_2014}. The method in ~\cite{ryu_extensions_2014} is comprised of two steps. The first step generates a quadrature rule with $M=N$ points, the existence of which is guaranteed by Tchakaloff's theorem~\cite{Tchakaloff_BSM_1957}. The second step uses this quadrature rule as an initial guess for a local gradient-based optimization procedure that searches for a quadrature rule with $M\le N$ points.  

The initial quadrature rule is generated by drawing a large number of points from the domain of integration and then solving an inequality-constrained linear program to find a quadrature rule with positive weights that matches all desired moments. Similar approaches can be found in ~\cite{Arnst_GPR_IJNME_2012,Constantine_PW_IJNME_2014}. The $N$ points comprising this initial quadrature rule are then grouped into $M\le N$ clusters. A new approximate quadrature rule is then formed by combining points and weights within a cluster into a single point and weight. 

Our numerical method differs from ~\cite{ryu_extensions_2014} in the following ways: (i) we use a different formulation of the linear program to generate the initial quadrature rule; (ii) we show numerically that this quadrature only needs to be solved with very limited accuracy, (iii) we present an automated way of selecting the clusters from the initial rule -- in  ~\cite{ryu_extensions_2014} no method for clustering points is presented; (iv) we provide extensive numerical testing of our method in a number of settings. 

The theoretical contributions of this paper include a lower bound on the number of points in the final quadrature rule, and a definition of quasi-optimality in terms of this bound. We also provide provide a simple means of testing whether the quadrature rules generated by any method are quasi-optimal.

The remainder of the paper is structured as following. In Section \ref{sec:optimal-quadrature}~we introduce some nomenclature, define quasi-optimality, and present a means to verify if a quadrature rule is quasi-optimal. We also use these theoretical tools to analytically derive quasi-optimal rule for additive functions. In Section~\ref{sec:algorithm} we detail out algorithm for generating quadrature rules. We then present a wide range of numerical examples, in Section~\ref{sec:numerical-results}, which explore the properties of the quadrature rules that our numerical algorithm can generate.

\subsection{Existing quadrature rules}
We give a concise description of some existing quadrature rules. There are numerous approaches for computing polynomial-based quadrature rules so our goal is not to be comprehensive, but instead to concentrate on rules that are related to our proposed method.

For univariate functions Gaussian quadrature is one of the most commonly used approaches. Nodes associated to Gaussian quadrature rules are prescribed by roots of the polynomials orthogonal to the measure $\mu$ \cite{szego_orthogonal_1975}. The resulting quadrature rule is always positive (meaning the weights $w_m$ are all positive) and the rules are optimal in the sense that given a Gaussian quadrature rule of degree of exactness $p$, no rule with fewer points can be used to integrate all degree-$p$ polynomials.

When integrating multivariate functions with respect to tensor-product measures on a hypercube accurate and efficient quadrature rules can be found by taking tensor-product of one-dimensional Gaussian quadrature rules. These rules will optimal for functions that can be represented exactly by tensor-product polynomial spaces of degree $p$. However the use of such quadrature rules is limited to a small number of dimensions, say 3-4, because the number of the points in the rule grows exponentially with dimension.

Sparse grid quadrature methods have been successfully used as an alternative to tensor-product quadrature for multivariate functions \cite{Gerstner_G_NA_1998,Gerstner_G_C_2003,Smolyak_SMD_1963}. The number of points in sparse grid rules grow logarithmically with dimension for a fixed level of accuracy. Unlike tensor-product rules however the quadrature weights will not all be positive. Sparse grid quadrature delays the curse of dimensionality by focusing on integrating polynomial spaces that have high-degree univariate terms but low-degree interaction terms. 

High-dimensional cubature rules can often be more effective than
sparse grid rules when integrating functions that are well represented
by total-degree polynomials. These rules have positive weights and
typically consist of a very small number of points. However such
highly effective cubature rules are difficult to construct and are
have only been derived for a specific set of measures, integration
domains and polynomial degree of exactness~\cite{Hammer_S_MTAC_1958,Stroud_book_71,Xiu_ANM_2008}.

Tensor-product integration schemes generally produce approximations to the integral \eqref{eq:quadrature} whose accuracy scales like $M^{-r/d}$, where $r$ indicates the maximum order of continuous partial derivatives in any direction \cite{Novak_R_inbook_1997}. This convergence rate illustrates both the blessing of smoothness -- regularity accelerates convergence exponentially -- along with the curse of dimensionality -- convergence is exponentially hampered by large dimension. For sparse grids consisting of univariate quadrature rules with $O(2^l)$ points have a similar error, scaling like $O(M^{-r}l^{(d-1)(r+1)})$~\cite{Gerstner_G_NA_1998}.


A contrasting approach is given by Monte Carlo (MC) and quasi Monte Carlo (QMC) approaches. These approximations produce convergence rates of $O(M^{-\frac{1}{2}})$ and $O(\log(M)^dM^{-1})$, respectively~\cite{Caflisch_AN_1998}. MC points are random, selected as independent and identically-distributed realizations of a random variable, and QMC points are deterministically generated as sequences that minimize discrepancy.




\section{Quasi-optimality in multivariate polynomial quadrature}\label{sec:optimal-quadrature}
The goal of this section is to mathematically codify relationships between the number of quadrature points $M$ and the dimension $N$ of the polynomial space $P_\Lambda$. In particular, we provide a theoretical lower bound for $M$ for a fixed $\Lambda$. This lower bound can be used to define a notion of optimality in polynomial quadrature rules. We also provide related characterizations of optimal quadrature rules in both one and several dimensions.

\subsection{Notation}\label{sec:notation}

With $d \geq 1$ fixed, we consider a positive measure $\mu$ on $\R^d$ with support $D = \mathrm{supp}\; \mu$. This support may be unbounded. The $L^2_\mu(D)$ inner product and norm are defined as
\begin{align*}
  \left\langle f, g \right\rangle_{\mu} &\coloneqq \int_D f(x) g(x) \dx{\mu}(x), & \|f\|_\mu^2 &= \left\langle f, f \right\rangle_\mu, & L^2_\mu(D) = \left\{ f: D \rightarrow \R \; | \; \|f\|^2_{\mu} < \infty \right\}
\end{align*}
To prevent degeneracy of polynomials with respect to $\mu$ and to ensure finite moments of $\mu$, we assume
\begin{align}\label{eq:mu-assumption}
  0 < \left\| p \right\|_\mu < \infty,
\end{align}
for all algebraic polynomials $p(x)$. One can guarantee the lower inequality if, for example, there is any open Euclidean ball in $D$ inside which $\mu$ has a density function.

A point $x$ in $\R^d$ has coordinate representation $x = \left(x^{(1)}, \ldots, x^{(d)} \right) \in \R^d$; for a multi-index $\alpha \in \N_0^d$ with coordinates $\alpha = \left(\alpha^{(1)}, \ldots, \alpha^{(d)}\right)$, we have $\alpha ! = \prod_{j=1}^d \alpha^{(j)} !$ and $x^\alpha = \prod_{j=1}^d \left[ x^{(j)}\right]^{\alpha^{(j)}}$. We let $\Lambda \subset \N_0^d$ denote a multi-index set of finite size $N$. 

If $\alpha, \beta \in \N_0^d$ are any two multi-indices and $k \in \R$, we define $\alpha + \beta$, $k \alpha$, and $\left\lfloor k \alpha \right\rfloor$ component wise. The partial ordering $\alpha \leq \beta$ is true if all the component wise conditions are true. We define the following standard properties and operations on multi-index sets:
\begin{definition}
  Let $\Lambda$ and $\Theta$ be two multi-index sets, and let $k \in [0, \infty)$.
  \begin{enumerate}[labelwidth=3.5cm,itemindent=1em,leftmargin=!]
    \item[(Minkowski addition)] The sum of two multi-index sets is
      \begin{align*}
        \Lambda + \Theta &= \left\{ \alpha + \beta \; |\; \alpha \in \Lambda,\; \beta \in \Theta \right\}
      \end{align*}
    \item[(Scalar multiplication)] The expression $k \Lambda$ is defined as
      \begin{align*}
        k \Lambda &= \left\{ k \alpha\; | \; \alpha \in \Lambda \right\}.
      \end{align*}
      Note that this need not be a set of multi-indices.
    \item[(Downward closed)] $\Lambda$ is a downward closed set if $\alpha \in \Lambda$ implies that $\beta \in \Lambda$ for all $\beta \leq \alpha$.
    \item[(Downward closure)] For any finite $\Lambda$, $\widebar{\Lambda}$ is the smallest downward closed set containing $\Lambda$,
      \begin{align*}
        \widebar{\Lambda} = \left\{ \alpha \in \N_0^d \; | \; \alpha \leq \beta \textrm{ for some } \beta \in \Lambda \right\}.
      \end{align*}
    \item[(Convexity)] $\Lambda$ is convex if for any $p \in [0, 1]$ and any $\alpha, \beta \in \Lambda$, then $\left\lfloor p \alpha + (1-p) \beta \right\rfloor \in \Lambda$.
  \end{enumerate}
\end{definition}
With our notation, scalar multiplication is not consistent with Minkowski addition. In particular we have $2 \Theta \subseteq \Theta + \Theta$ in general. If $\Lambda$ is downward closed, then $\widebar{\Lambda} = \Lambda$. 

The polynomial space $P_\Lambda$ is defined by a given multi-index $\Lambda$:
\begin{align*}
  P_\Lambda &= \mathrm{span} \left\{ x^\alpha \; | \; \alpha \in \Lambda \right\}, & |\Lambda| &= N,
\end{align*}
and $P_\Lambda$ has dimension $N$ in $L^2_\mu(D)$ under the assumption \eqref{eq:mu-assumption}. Note that we make no particular assumptions on the structure of $\Lambda$. I.e., we do not assume $\Lambda$ is downward closed, but much of our theory and all our numerical examples use downward closed index sets. 

On $\N_0^d$, we will make use of the $\ell^p$ norm $\left\|\cdot\right\|_p$ for $0 \leq p \leq \infty$, and the associated ball $B_{p}(r)$ of radius $r \geq 0$ to define index sets. These sets are defined by
\begin{align*}
  B_{p}(r) = \left\{ \alpha \in \N_0^d \; | \; \left\|\alpha \right\|_p \leq r \right\}.
\end{align*}
The $\ell^p$ norms are defined for $p=0$, $0 < p < \infty$, and $p = \infty$ by, respectively,
\begin{align*}
 \left\| \alpha \right\|_0 &= \sum_{j=1}^d \mathbbm{1}_{\alpha_j \neq 0}, & \left\| \alpha \right\|_p^p &= \sum_{j=1}^d \alpha_j^p, & \left\| \alpha \right\|_\infty &= \max_{1 \leq j \leq d} \alpha_j.
\end{align*}
The index sets sets $B_{p}(r)$ are all downward closed, and are convex if $p \geq 1$. The set $B_0(r)$ equals $\N_0^d$ when $r \geq d$.




\subsection{Quasi-optimal quadrature}\label{sec:quasi-optimal quadrature}

With $M$ fixed, the theoretical and computational tractability of computing a solution to \eqref{eq:quadrature-condition} depends on $\Lambda$, $D$, and $\mu$. In particular, it is unreasonable to expect that $\Lambda$ can be arbitrarily large; if this were true then $\mu$ can be approximated to arbitrary accuracy by a sum of $M$ Dirac delta distributions, which would allow us to violate the lower inequality in \eqref{eq:mu-assumption}. There is a strong heuristic that motivates the possible size of $\Lambda$: The set $\left\{ x_1, \ldots, x_M \right\}$ represents $M d$ degrees of freedom, and varying $w_j$ ($j=1, \ldots, M$) represents an additional $M$ degrees of freedom. For an $N$-dimensional space $P_\Lambda$, \eqref{eq:quadrature-condition} can be ensured with $N$ constraints. Thus we expect for general $(\mu,D)$ that $\Lambda$ (and thus $N$) must be small enough to satisfy
\begin{align}\label{eq:counting-heuristic}
  |\Lambda| = N \leq (d + 1) M.
\end{align}
We will show that this heuristic does not always produce a faithful bound on sizes of quadrature rules; we provide instead a strict lower bound on the number of points $M$ in a quadrature rule for a given $\Lambda$.
To proceed we require the notion of `half-sets'.
\begin{definition}
  Let $\Lambda \in \N_0^d$ be a finite, nontrivial, downward-closed set. A multi-index set $\Theta$ is 
  \begin{enumerate}
    \item a half-set for $\Lambda$ if $\Theta + \Theta \subseteq \Lambda$
    \item a maximal half-set for $\Lambda$ if it is a half-set of maximal size. I.e, if $|\Theta| = L$, with 
      \begin{align}\label{eq:L-definition}
        L = L(\Lambda) &= \max \left\{ |\Theta|\; | \; \Theta \textrm{ a multi-index set satisfying } \Theta + \Theta \subseteq \Lambda \right\}
      \end{align}
  \end{enumerate}
  We call $L$ the \textit{maximal half-set size} of $\Lambda$.
\end{definition}
Recall that $2 \Theta \subseteq \Theta + \Theta$ so that the terminology ``half" should not be conflated with the operation of halving each index in an index set. If $\Lambda$ is not downward closed, it may not have any half sets. However, all nontrivial downward-closed sets have at least one nontrivial half-set (the zero set $\left\{ 0\right\}$ is one such half set). Thus maximal half sets always exist in this case, but they are not necessarily unique. An example in $d=2$ illustrates this non-uniqueness:
\begin{gather*}
  \Lambda = B_0(1) \cap B_1(2) = \left\{ (0,0),\; (0,1),\; (0,2),\; (1,0),\; (2,0) \right\}, \\
  \Theta_1 = \left\{ (0,0),\; (1,0) \right\}, \hskip 10pt  \Theta_2 = \left\{ (0,0),\; (0,1) \right\}.
\end{gather*}
We have $L(\Lambda) = 2$, and both $\Theta_1$ and $\Theta_2$ are maximal half sets for $\Lambda$. 

If $\Lambda$ is both downward-closed and convex, then its maximal half-set is unique and easily computed.
\begin{theorem}
  Let $\Lambda$ be convex and downward-closed. Then its maximal half-set $\Theta$ is unique, given by $\Theta = \left\lfloor \frac{1}{2} \Lambda \right\rfloor$.
\end{theorem}
\begin{proof}
  Let $\Theta$ be any half-set for $\Lambda$. Then for any $\theta \in \Theta$, we have $2 \theta \in \Lambda$, so that $\theta \in \left\lfloor \frac{1}{2} \Lambda \right\rfloor$. Thus $\Theta \subseteq \left\lfloor \frac{1}{2} \Lambda \right\rfloor$, showing that any half-set must be contained in $\left\lfloor \frac{1}{2} \Lambda \right\rfloor$.

  Now let $\theta_1, \theta_2 \in \left\lfloor \frac{1}{2} \Lambda  \right\rfloor$. Then $2 \theta_1, 2\theta_2 \in \widebar{\Lambda} = \Lambda$.  Thus, by convexity
  \begin{align*}
    \theta_1 + \theta_2 &= \frac{1}{2} (2 \theta_1) + \frac{1}{2} (2 \theta_2) \in \Lambda,
  \end{align*}
  where the set inclusion holds by convexity of $\Lambda$. Thus, $\left\lfloor \frac{1}{2} \Lambda  \right\rfloor$ is itself a half-set; by the previous observation that it also dominates any half-set, then it must be the unique largest (maximal) half-set.
\end{proof}

We can now state one of the main results of this section: The number $L(\Lambda)$ in \eqref{eq:L-definition} is a lower bound on the size of any quadrature rule satisfying \eqref{eq:quadrature-condition}.
\begin{theorem}\label{lemma:N-condition}
  Let $\Lambda$ be a finite downward-closed set, and suppose that an $N$-point quadrature rule $\left\{x_j, w_j \right\}_{n=1}^N$ exists satisfying \eqref{eq:quadrature-condition}. Then $N \geq L(\Lambda)$, with $L$ the maximal half-set size defined in \eqref{eq:L-definition}.
\end{theorem}
\begin{proof}
  Let $\Theta$ be any set satisfying $\Theta + \Theta \subseteq \Lambda$ and $|\Theta| = L$, with $L$ defined in \eqref{eq:L-definition}. We choose any size-$L$ $\mu$-orthonormal basis for $P_\Theta$:
  \begin{align*}
    P_\Theta &= \mathrm{span} \left\{ q_j\right\}_{j=1}^L, & \int q_j(x) q_k(x) \dx{\mu(x)} &= \delta_{k,j}
  \end{align*}
  Note that $q_j$ and $q_k$ have monomial expansions
  \begin{align*}
    q_j(x) &= \sum_{\alpha \in \Theta} c_\alpha x^\alpha, &
    q_k(x) &= \sum_{\alpha \in \Theta} d_\alpha x^\alpha, 
  \end{align*}
  for fixed $j,k=1, \ldots, N$ and some constants $c_\alpha$ and $d_\alpha$. This implies
  \begin{align*}
    q_j(x) q_k(x) = \sum_{\alpha, \beta \in \Theta} c_\alpha d_\beta x^{\alpha + \beta} = \sum_{\alpha \in (\Theta + \Theta)} f_\alpha x^\alpha.
  \end{align*}
  Since $\Theta$ is a half set for $\Lambda$, then $q_j q_k \in P_\Lambda$ and is therefore exactly integrated by the quadrature rule \eqref{eq:quadrature-condition}.

  Then consider the $L \times L$ matrix $\bs{G}$ defined as
  \begin{align*}
    \bs{G} &= \bs{V}^T \bs{W} \bs{V}, & (G)_{j,k} &= \sum_{n=1}^N w_n q_j(x_n) q_k(x_n) \\
    (V)_{j,k} &= q_k(x_j), & (W)_{j,k} &= w_j \delta_{j,k}.
  \end{align*}
  The matrices $\bs{V}$ and $\bs{W}$ are $N \times L$ and $N \times N$, respectively. Since the quadrature rule exactly integrates $q_j q_k$, then $\bs{G} = \bs{I}$, the $L \times L$ identity matrix. Thus, the matrix product $\bs{V}^T \bs{W} \bs{V}$ has rank $L$. If $N < L$ then, e.g., $\mathrm{rank}(\bs{W}) < L$ and so it is not possible that $\mathrm{rank} \left(\bs{V}^T \bs{W} \bs{V}\right) = L$.
\end{proof}
This result shows that if $N < L(\Lambda)$, then an $N$-point quadrature rule satisfying \eqref{eq:quadrature-condition} cannot exist. This is nontrivial information. As an example, let $d=2$, and consider
\begin{align*}
\Lambda = B_{\infty}(2) = \left\{ (0,0), (0,1), (0,2), (1,0), (1,1), (1,2), (2,0), (2,1), (2,2) \right\},
\end{align*}
corresponding to the tensor-product space of degree 2. Since $|\Lambda| = 9$, the heuristic \eqref{eq:counting-heuristic} suggests that it is possible to find a rule with only 3 points. However, let $\Theta = \left\{ (0,0), (0,1), (1,0), (1,1) \right\}$. Then $\Theta + \Theta = \Lambda$ and $|\Theta| = 4$. Thus, no 3-point rule that is accurate on $P_\Lambda$ can exist.

Another observation from Lemma \ref{lemma:N-condition} is that we may justifiably call an $N$-point quadrature rule optimal if $N = L$, in the sense that one cannot achieve the same accuracy with a smaller number of nodes.
\begin{definition}
  Let $\Lambda$ be a downward-closed multi-index set. We call an $M$-point quadrature rule \textit{quasi-optimal} for $\Lambda$ if it satisfies \eqref{eq:quadrature-condition} with $M=L(\Lambda)$, with $L$ defined in \eqref{eq:L-definition}.
\end{definition}
We call these sets \underline{quasi}-optimal because, given a quasi-optimal rule for $\Lambda$, it may be possible to generate a quadrature rule of equal size that is accurate on an index set that is strictly larger than $\Lambda$ (see Section \ref{sec:poly-spaces}). Quasi-optimal quadrature rules are not necessarily unique and sometimes do not exist. However, the weights for quasi-optimal quadrature rules have a precise behavior. 

Under the assumption \eqref{eq:mu-assumption}, we can find an orthonormal basis $q_j$ for $P_\Lambda$:
\begin{align*}
  \left\langle q_j, q_k \right\rangle_\mu &= \delta_{j,k}, & P_\Lambda = \mathrm{span} \left\{ q_j \right\}_{j=1}^N
\end{align*}
Given this orthonormal basis, define
\begin{align*}
  \lambda_\Lambda(x) &= \frac{1}{\sum_{j=1}^N q_j^2(x)}.
\end{align*}
The quantity $\lambda_\Lambda$ depends on $D$, $\mu$, and $P_\Lambda$, but not on the particular basis $q_j$ we have chosen: Through a unitary transform we may map $\left\{ q_j \right\}_{j=1}^N$ into any other orthonormal basis for $P_\Lambda$, but this transformation leaves the quantity above unchanged. The weights for a quasi-optimal quadrature rule on $\Lambda$ are evaluations of $\lambda_\Theta$, where $\Theta$ is a maximal half-set for $\Lambda$.
\begin{theorem}\label{lemma:w-condition}
  Let $\Lambda$ be a finite downward-closed set, and let an $M$-point quadrature rule $\left\{x_m, w_m \right\}_{m=1}^M$ be quasi-optimal for $\Lambda$. Then
  \begin{align}\label{eq:w-condition}
    w_j &= \lambda_\Theta(x_j), & j&=1, \ldots, N,
  \end{align}
  where $\Theta$ is a(ny) maximal half set for $\Lambda$.
\end{theorem}
\begin{proof}
  Let $\ell_j(x)$, $j=1, \ldots, N$, be the cardinal Lagrange interpolants from the space $P_{\Theta}$ on the nodes $\left\{x_j\right\}_{j=1}^N$.\footnote{The interpolation problem of $x_j$ for any basis of $P_\Theta$ must be invertible since the Vandermonde-like matrix $\bs{V}$ in the proof of Lemma \ref{lemma:N-condition} must have full rank, be square since $N=L$, and thus be invertible. Thus, the $\ell_j$ are well-defined.} These cardinal interpolants satisfy $\ell_j(x_k) = \delta_{j,k}$. The weights $w_j$ must be positive since
  \begin{align}\label{eq:w-positive}
    w_j = \sum_{k=1}^N w_k \ell_j^2(x_k) = \int \ell_j^2(x) \dx{\mu(x)} > 0,
  \end{align}
  where the second equality uses the fact that the quadrature rule is exact on $P_\Lambda \ni \ell_j^2$. By \eqref{eq:w-positive}, then $\sqrt{\bs{W}}$ is well-defined. With $\bs{V}$ and $\bs{W}$ defined in the proof of Lemma \ref{lemma:N-condition}, then $\bs{V}^T \bs{W} \bs{V} = \bs{I}$, and both matrices are square ($N = L$). Thus, $\sqrt{\bs{W}} \bs{V}$ is an orthogonal matrix, so that
  \begin{align*}
    \sqrt{\bs{W}} \bs{V} \bs{V}^T \sqrt{\bs{W}^T} &= \bs{I} \\
    \bs{V} \bs{V}^T &= \bs{W}^{-1}
  \end{align*}
  Taking the diagonal components of the left- and right-hand side shows that $w_j = \left( \sum_{k=1}^N (V)^2_{j,k} \right)^{-1}$.
\end{proof}

\section{Theoretical counter/examples of quasi-optimal quadrature}\label{sec:conseequences}
To provide further insight into our definition of quasi-optimality, we
investigate the consequences of the above theoretical results through some examples.

\subsection{Univariate rules}\label{sec:univariate}
The example in this section shows (i) that optimal quadarture rules in one dimension are Gauss quadrature rules, and (ii) among many quasi-optimal rules, some may be capable of accurately integrating more polynomials than the others.

Let $d = 1$ with $\Lambda = \left\{ 0, \ldots, N-1\right\}$. The maximal half-set for $\Lambda$ is $\Theta = \left\{ 0, \ldots, \left\lfloor\frac{N-1}{2} \right\rfloor\right\}$, and thus 
\begin{align*}
  L(\Lambda) = \left\lfloor \frac{N-1}{2} \right\rfloor + 1
\end{align*}
We consider two cases, that of even $N$, and of odd $N$.

Suppose $N$ is even. Then a quasi-optimal quadrature rule has $M = |\Theta| = \frac{N}{2}$ abscissae, and exactly integrates $M = 2N$ (linearly independent) polynomials. I.e., the $N$-point rule exactly integrates polynomials up to degree $2 M - 1$. It is well-known that this quadrature rule is unique: it is the $\mu$-Gauss-Christoffel quadrature rule \cite{szego_orthogonal_1975}. 

With $\left\{q_j\right\}_{j=0}$ a family of $L^2_\mu(D)$-orthonormal polynomials, with $\deg q_j = j$, then the abscissae $x_1, \ldots, x_M$ are the zeros of $q_M$, and the weights are given by \eqref{eq:w-condition}. This quadrature rule can be efficiently computed from eigenvalues and eigenvectors of the symmetric tridiagonal Jacobi matrix associated with $\mu$ \cite{golub_calculation_1969}. Thus, this quadrature rule can be computed without resorting to optimization routines.

Now suppose $N$ is odd. Then a quasi-optimal quadrature rule has $M = L(\Lambda) = \frac{N+1}{2}$ abscissae, and exactly integrates $2 M - 1$ polynomials. I.e., the $M$-point rule exactly integrates polynomials up to degree $2 M - 2$. Clearly a Gauss-Christoffel rule is a quasi-optimal rule in this case. However, by our definition there are an (uncountably) \textit{infinite} number of quasi-optimal rules in this case \cite{szego_orthogonal_1975}. In particular, for arbitrary $c \in \R$, the zero set of the polynomial
\begin{align}\label{eq:gauss-radau}
  q_M(x) - c q_{M-1}(x),
\end{align}
corresponds to the abscissae of a quasi-optimal quadrature rule, with the weights give by \eqref{eq:w-condition}. Like the Gauss-Christoffel case, computation of these quadrature rules can be accomplished via eigendecompositions of symmetric tridiagonal matrices \cite{narayan2017}.

In classical scientific computing scenarios, sets of $N$-point rules with polynomial accuracy of degree $2 N -2$ have been called Gauss-Radau rules, and are traditionally associated to situations when $\supp \mu$ is an interval and one of the quadrature absciasse is collocated at one of the endpoints of this interval \cite{golub_modified_1973}.


\subsection{Quasi-optimal multivariate quadrature}\label{sec:additive-set-example}
This section furnishes an example where a quasi-optimal multivariate quadrature rule can be explicitly constructed. 

Consider a tensorial domain and probability measure, i.e., suppose
\begin{align*}
  \mu &= \times_{j=1}^d \mu_j, & D &= \otimes_{j=1}^d \supp \mu_j
\end{align*}
where $\mu_j$ are univariate probability measures. We let $\Lambda = B_{0}(1) \cap B_{1}(n)$ for any $n \geq 2$. With $e_j$ the cardinal $j$'th unit vector in $\N_0^d$, then $\Lambda$ is the set of indices of of the form $q e_j$ for $0 \leq q \leq n$; the size of $\Lambda$ is $N = d n + 1$. Note that this $\Lambda$ would arise in situations where one seeks to approximate an unknown multivariate function as a sum of univariate functions; this is rarely a reasonable assumption in practice. However, in this case we can explicitly construct optimal quadrature rules.

The heuristic \eqref{eq:counting-heuristic} suggests that we can construct a rule satisfying \eqref{eq:quadrature-condition} if we use
\begin{align*}
  M \geq 
  n \left(\frac{d}{d+1}\right) + \frac{1}{d+1}
\end{align*}
nodes, which is approximately $n$ nodes. However, we can achieve this with fewer nodes, only $M = \left\lfloor n/2 \right\rfloor + 1$, independent of $d$. Here the heuristic \eqref{eq:counting-heuristic} is too pessimistic when $d \geq 2$. 

Associated to each $\mu_j$, we need the corresponding system of orthonormal polynomials $q$ and the univariate $\lambda$ function. For each $j=1, \ldots, d$, let $q_{n,j}(x)$, $n=0, 1, \ldots,$ denote a $L^2_{\mu_j}$-orthonormal polynomial family with $\deg q_{n,j} = n$. Define 
\begin{align*}
  \lambda_{n,j}\left(\cdot \right) = \frac{1}{\sum_{i=0}^{n} q_{i,j}^2(\cdot)}.
\end{align*}
Note that $\lambda_{0,j} = 1/q_{0,j}^2 \equiv 1$ for all $j$ since each $\mu_j$ is a probability measure.


Consider the index sets $\Theta_j = \left\{ 0, e_j, 2e_j, \ldots, \left\lfloor n/2 \right\rfloor  e_j \right\}$, for $j = 1, \ldots, d$, where $0$ is the origin in $\N_0^d$. Each $\Theta_j$ is a maximal half-set for $\Lambda$, and $L$ in \eqref{eq:L-definition} is given by $L = \left\lfloor n/2 \right\rfloor+1$.


To construct a quasi-optimal quadrature rule achieving with $M = L$ nodes, we note that the weights $w_j$, $j=1, \ldots, N$ must be given by \eqref{eq:w-condition}, which holds for \textit{any} maximal half index set $\Theta$. I.e., it must simultaneously hold for \textit{all} $\Theta_j$. Thus, 
\begin{align}\label{eq:christoffel-equivalence}
  w_m = \lambda_{{\Theta_j}}\left(x_m\right) = \left[ \sum_{i=0}^{n} q_{i,j}^2\left(x_m^{(j)}\right) \prod_{\substack{s = 1, \ldots, d\\s \neq j}} q^2_{0,s} \right]^{-1} = \lambda_{\left\lfloor n/2 \right\rfloor,j} \left( x_m^{(j)}\right),
\end{align}
for $j = 1, \ldots, d$. This implies in particular that the coordinates $x_m^{(j)}$ for node $m$ must satisfy
\begin{align*}
 \lambda_{\left\lfloor n/2 \right\rfloor,1} \left( x_m^{(1)}\right) = 
 \lambda_{\left\lfloor n/2 \right\rfloor,2} \left( x_m^{(2)}\right) = \cdots 
 = \lambda_{\left\lfloor n/2 \right\rfloor,d} \left( x_m^{(d)}\right).
\end{align*}
We can satisfy this condition in certain cases. Suppose $\mu_1 = \mu_2 = \cdots = \mu_d$; then $\lambda_{n,1} = \lambda_{n,2} = \cdots = \lambda_{n,d}$ and so we can satisfy \eqref{eq:christoffel-equivalence} by setting $x^{(1)}_m = x^{(2)}_m = \cdots = x^{(d)}_m$ for all $m = 1, \ldots, M$. Thus nodes for a quasi-optimal quadrature rule nodes could lie in $\R^d$ along the graph of the line defined by
\begin{align*}
  x^{(1)} = x^{(2)} = \cdots = x^{(d)}.
\end{align*}
In order to satisfy the integration condition \eqref{eq:quadrature-condition} we need distribute the nodes in an appropriate way. Having effectively reduced the problem to one dimension, this is easily done: we choose a Gauss-type quadrature rule as described in the previous section. Let the $j$th coordinate of the quadrature rule be the $M$-point Gauss quadrature nodes for $\mu_j$, i.e.,
\begin{align*}
  \left\{ x_1^{(j)}, x_2^{(j)}, \ldots, x_M^{(j)} \right\} = q_{M,j}^{-1}(0).
\end{align*}
This then uniquely defines $x_1, \ldots, x_M$, and $w_m$ is likely uniquely defined since we have satisfied \eqref{eq:christoffel-equivalence}.

Thus a ``diagonal" Gauss quadrature rule, $M = \left\lfloor n/2 \right\rfloor + 1$, with equal coordinate values for each abscissa, is an optimal rule in this case. 

\subsection{Quasi-optimal multivariate quadrature: non-uniqeness}
The previous example with an additional assumption allows to construct $2^{d-1}$ distinct quasi-optimal quadrature rules. Again take $\Lambda = B_{0}(1) \cap B_{1}(n)$ for $n\geq 2$, and let $\mu = \times_{j=1}^d \mu_j$ with identical univariate measures $\mu_j = \mu_k$. 

To this add the assumption that $\mu_j$ is a symmetric measure; i.e., for any set $A \subset \R$, then $\mu_j(A) - \mu_j(-A)$. In this case the univariate orthogonal polynomials $q_{k,j}$ are even (odd) functions if $k$ is even (odd). Thus, the set zero set $q_{k,j}^{-1}(0)$ is symmetric around 0, and $\lambda_{M,j}$ is always an even function. With $x_m$ the quasi-optimal set of nodes defined in the previous section, let 
\begin{align*}
  y_m^{(j)} &= \sigma_j x_m^{(j)}, & \sigma_j \in \left\{ -1, +1 \right\},
\end{align*}
for a fixed but arbitrary sign train $\sigma_1, \ldots, \sigma_d$. Using the above properties, one can show that the nodes $\left\{y_1, \ldots, y_M\right\}$ and the weights $w_1, \ldots, w_M$ define a quadrature rule satisfying \eqref{eq:quadrature-condition}, and of course have the same number of nodes as the quasi-optimal rule from the previous section.

By varying the $\sigma_j$, we can create $2^{d-1}$ unique distributions of nodes, thus showing that at least this many quasi-optimal rules exist.

\subsection{Quasi-optimal multivariate quadrature: non-existence}
We again use the setup of Section \ref{sec:additive-set-example}, but this time to illustrate that it is possible for quasi-optimal to not exist. We consider $d=2$, and take $\Lambda = B_{0}(1) \cap B_{1}(3)$, and let $\mu = \mu_1 \times \mu_2$ for two univariate probability measures $\mu_1$ and $\mu_2$. However, this time we let these measures be different:
\begin{align*}
  \dx{\mu_1}(t) &= \frac{1}{2} \dx{t}, & \mathrm{supp}\;\mu_1 &= [-1,1] \\
  \dx{\mu_2}(t) &= \frac{3}{4} (1-t^2) \dx{t}, & \mathrm{supp}\;\mu_2 &= [-1,1]
\end{align*}
With our choice of $\Lambda$, we have $L(\Lambda) = 2$ with two maximal half-sets:
\begin{align*}
  \Theta_1 &= \left\{ (0,0), (1,0) \right\}, & \Theta_2 &= \left\{ (0,0), (0,1) \right\}
\end{align*}
In this simple case the explicit $\mu_1$- and $\mu_2$-orthonormal families take the expressions
\begin{align*}
  q_{0,1}(t) &= 1, & q_{1,1}(t) &= \sqrt{3} t, \\
  q_{0,2}(t) &= 1, & q_{1,2}(t) &= \frac{\sqrt{5}}{2} t
\end{align*}
so that associated to $\Theta_1$ and $\Theta_2$, respectively, we have the functions
\begin{align*}
  \lambda_{1,1}(t) &= \frac{1}{1 + 3 t^2}, & 
  \lambda_{1,2}(t) &= \frac{4}{4 + 5 t^2}.
\end{align*}
Since by \eqref{eq:christoffel-equivalence} we require $w_m = \lambda_{1,1}\left(x_m^{(1)}\right) = \lambda_{1,2}\left(x_m^{(2)}\right)$, this implies that 
\begin{subequations}
\begin{align}\label{eq:noexist-a}
  3 \left( x_m^{(1)}\right)^2 &= \frac{5}{4} \left( x_m^{(2)} \right)^2, & m = 1, 2
\end{align}
However, the condition \eqref{eq:quadrature-condition} also implies for our choice of $\Lambda$ that
\begin{align*}
  \int_{-1}^2 p(t) \frac{1}{2} \dx{t} &= \sum_{m=1}^2 w_m p\left(x^{(1)}_m\right), & p &\in \mathrm{span} \left\{ 1, t, t^2, t^3 \right\}, \\
  \int_{\R} p(t) \frac{3}{4}(1-t^2)  \dx{t} &= \sum_{m=1}^2 w_m p\left(x^{(2)}_m\right), & p &\in \mathrm{span} \left\{ 1, t, t^2, t^3 \right\}.
\end{align*}
The conditions above imply that $\left\{x^{(1)}_m\right\}_{m=1}^2$ and $\left\{x^{(2)}_m\right\}_{m=1}^2$ be nodes for the 2-point $\mu_1$- and $\mu_2$-Gauss quadrature rules, respectively, which are both unique. Thus, $\left\{x^{(1)}_m\right\}_{m=1}^2 = \left\{ \pm \sqrt{3}/3 \right\}$, and $\left\{ x^{(2)}_m\right\}_{m=1}^2 = \left\{ \pm 1/\sqrt{5} \right\}$. Thus, we have the equality
\begin{align}\label{eq:noexist-b}
  3 \left( x_m^{(1)}\right)^2 &= 5 \left( x_m^{(2)} \right)^2, & m = 1, 2
\end{align}
\end{subequations}
We arrive at a contradiction: simultaneous satisfaction of \eqref{eq:noexist-a} and \eqref{eq:noexist-b} implies all coordinates of the abscissae are 0, which cannot satisfy \eqref{eq:quadrature-condition}. Thus, no quasi-optimal quadrature rule for this example exists.

\section{Numerically generating multivariate quadrature rules}
\label{sec:algorithm}
In this section we describe our proposed algorithm for generating multivariate quadrature rules; the algorithm generates rules having significantly less points than the number of moments being matched. We will refer to such rules as \textit{reduced quadrature} rules. We will compare the number of nodes our rules can generate with the lower bound $L(\Lambda)$ defined in Section \ref{sec:optimal-quadrature}, along with the heuristic bound \eqref{eq:counting-heuristic}.

Our method is an adaptation of the method presented in \cite{ryu_extensions_2014}. The authors there showed that one can recover a univariate Gauss quadrature rule as the solution to an infinite-dimensional linear program (LP) over nonnegative measures.  Let a finite index set $\Lambda$ be given with $|\Lambda| = {N}$, and suppose that $\left\{ p_j \right\}_{j=1}^{N}$ is any basis for $P_\Lambda$. In addition, let $r$ be a polynomial on $\R^d$ such that $r \not\in P_\Lambda$; we seek a positive Borel measure $\nu$ on $\R^d$ solving
\begin{align}\label{eq:inifite-lp}
  \textrm{minimize } &\int r(\V{x}) \dx{\nu(\V{x})} \\
  \textrm{subject to } &\int p_j(\V{x}) \dx{\nu(\V{x})} = \int p_j(\V{x}) \dx{\mu(\V{x})}, \hskip 5pt j=1, \ldots, {N_{\Lambda}}
\end{align}
The restriction that $\nu$ is a positive measure will enter as a constraint into the optimization problem. A nontrivial result is that a positive measure solution to this problem exists with $|\supp \nu| = M \leq {N}$. Such a solution immediately yields a positive quadrature rule with nodes $\left\{\V{x}_{1}, \ldots, \V{x}_{M} \right\} = \supp \nu$ and weights given by $w_j = \nu(\V{x}_{j})$. Unfortunately, the above optimization problem is NP-hard, and so the authors in \cite{ryu_extensions_2014} propose a two step procedure for computing an approximate solution. The first step solves a finite-dimensional linear problem similar to~\eqref{eq:inifite-lp} to find a $K$-point positive quadrature rule with $K \leq N$. In the second step, the $K\leq{N}$ points are clustered into $M$ groups, where $M$ is automatically chosen by the algorithm. The clusters are then used to form a set of $M$ averaged points and weights, which form an approximate quadrature rule. This approximate quadrature rule is refined using a local gradient-based method to optimize a moment-matching objective. The precise algorithm is outlined in detail in Appendix~\ref{app:lp-algorithm}.

In this paper we also adopt a similar two step procedure to compute reduced quadrature rules. We outline these two steps in detail in the following section. Pseudo code for generating reduced quadrature rules is presented in Algorithm~\ref{alg:reduced-quadrature} contained in Appendix~\ref{app:algorithms}. 

\subsection{Generating an initial condition}\label{sec:algorithm-initial}
Let an index set $\Lambda$ be given with $|\Lambda| = {N} < \infty$, and suppose that $\left\{ p_j \right\}_{j=1}^n$ is a basis for $P_\Lambda$.
We seek to find a discrete measure $\nu=\sum_k v_k\delta_{\V{x}^{(k)}}$ by choosing a large candidates mesh, $X_S = \left\{ x_k \right\}_{k=1}^S$, with $S \gg N$ points on $D$ and solving the $\ell_1$ minimization problem
    \begin{align}\label{eq:l1-problem}
\begin{split}
      \textrm{minimize } &\sum_{k=1}^S |v_k| \\
      \textrm{subject to } &\sum_{k=1}^S v_k p_j({x}_k) = \int p_j({x}) \dx{\mu({x})}, \hskip 5pt j=1, \ldots, {N_{\Lambda}} \\
      \textrm{and } & v_k \geq 0, \hskip 5pt k=1, \ldots, S
\end{split}
    \end{align}
    The optimization is over the $S$ scalars $v_k$. With $\bs{v} \in \R^S$ a vector containing the $v_k$, the non-zero coefficients then define a quadrature rule with $K=\left\|\bs{v}\right\|_0$ points, $K \leq N$. The points corresponding to the non-zero coefficients $v_k$ are the quadrature points and the coefficients themselves are the weights.

This $\ell_1$-minimization problem as well as the residual based linear program used by~\cite{ryu_extensions_2014} become increasingly difficult to solve as the size of $\Lambda$ and dimension $d$ of the space increase. The ability to find a solution is highly sensitive to the candidate mesh. Through extensive experimentation we found that by solving~\eqref{eq:l1-problem} approximately via 
    \begin{align}\label{eq:bpdn-problem}
\begin{split}
      \textrm{minimize } &\sum_{k=1}^S |\alpha_k| \\
      \textrm{subject to } &\lvert \sum_{k=1}^S \alpha_k p_j(\V{x}^{(k)}) - \int p_j(\V{x}) \dx{\mu(\V{x})}\rvert<\epsilon, \hskip 5pt j=1, \ldots, n \\
      \textrm{and } & \alpha_k \geq 0, \hskip 5pt k=1, \ldots, M
\end{split}
    \end{align}
    for some $\epsilon >0$, we were able to find solutions that, although possibly inaccurate with respect to the moment matching criterion, could be used as an initial condition for the local optimization to find an accurate reduced quadrature rule. This is discussed further in Section~\ref{sec:tp-measures}. We were not able to find such approximate initial conditions using the linear program used in~\cite{ryu_extensions_2014}; we show results supporting this in Section \ref{sec:numerical-results}.

We solved~\eqref{eq:bpdn-problem} using least angle regression with a
LASSO modification~\cite{Tibshirani_JRSSBM_1996} whilst enforcing 
positivity of the coefficients. This algorithm iteratively adds and
removes positive weights $\alpha_k$ until $\epsilon=0$, or until no new
weights can be added without violating the positivity constraints. This
allows one to drive $\epsilon$ to small values without requiring an \textit{a priori} estimate of $\epsilon$.

In our examples, the candidate mesh ${X}_S$ is selected by generating uniformly random samples over the integration domain $D$, regardless of the integration measure. Better sampling strategies for generating the candidate mesh undoubtedly exist, but these strategies will be $D$- and $\Lambda$-dependent, and are not the central focus of this paper. Our limited experimentation here suggested that there was only marginal benefit from exploring this issue.

\subsection{Finding a reduced quadrature rule}
Once an initial condition has been found by solving ~\eqref{eq:bpdn-problem} we then use a simple greedy clustering algorithm to generate an initial condition for a local gradient-based moment-matching optimization.

\subsubsection{Clustering}
The greedy clustering algorithm finds the point with the smallest weight and combines it with its nearest neighbor. These two points are then replaced by a point whose coordinates are a convex combination of the two clustered points, where the convex weights correspond to the $v_k$ weights that are output from the LP algorithm. The newly clustered point set has one less point than before. This process is repeated. , 

We terminate the clustering iterative procedure when a desired number of points $M$ is reached. At the termination of the clustering algorithm a set of points $\widehat{{x}}_m$ and weights $\widehat{w}_m$, defining an approximate quadrature rule are returned. The algorithm pseudo-code describing this greedy clustering algorithm is given in Algorithm~\ref{alg:cluster} in Appendix~\ref{app:algorithms}.

The ability to find an accurate reduced quadrature rule is dependent on the choice of the number of specified clusters $M$. As stated in Section~\ref{sec:optimal-quadrature}, there is a strong heuristic that couples the dimension $d$, the number of matched moments $N$, and the number of points $M$. For general $\mu$ and given ${N}=|\Lambda|$ we set the number of clusters to be
\begin{align}\label{eq:dof-heuristic}
  M = \frac{N}{d+1}.
\end{align}
As shown in Section~\ref{sec:optimal-quadrature} this heuristic will not always produce a quadrature rule that exactly integrates all specified moments. Moreover, the heuristic may over estimate the actual number of points in a quasi-optimal quadrature rule. 

It is tempting to set $M$ to the lower bound value $L(\Lambda)$, defined in \eqref{eq:L-definition}. However, a sobering result of our experimentationl is that in all our experiments we were never able to numerically find a quadrature rule with fewer points than that specified by \eqref{eq:dof-heuristic}; therefore, we could not find any rules with $M$ points where $L(\Lambda) \leq M < N/(d+1)$. However, we were able to identify situations in which the heuristic underestimated the requisite number of points in a reduced quadrature rule. For example, if one wants to match the moments of a total degree basis of degree $k$ one must use at least $ k/2+d\choose d$ points. This lower bound is typically violated by the heuristic for low-degree $k$ and high-dimension $d$. E.g. for $d=10$ and $k=2$ we have $M=6$ using the heuristic yet the theoretical lower bound for $M$ from Theorem~\ref{lemma:N-condition} requires $M \geq L(\Lambda)=11$. In this case our theoretical analysis sets a lower bound for $M$ that is more informative than the heuristic \eqref{eq:dof-heuristic}.

\subsubsection{Local-optimization}
The approximate quadrature rule $\widehat{{x}}_k$, $\widehat{w}_k$ 
generated by the clustering algorithm is used as an initial guess for the following local optimization
\begin{subequations}
    \begin{align}\label{eq:local-optimization}
      \textrm{minimize } &\sum_{j=1}^{N} \left( \int p_j({x}) \dx{\mu({x})} - \sum_{m=1}^M w_m p_j({x}_{m}) \right)^2 \\\label{eq:local-constraints}
      \textrm{subject to } & {x}_{m} \in D \textrm{ and } w_m \geq 0,
    \end{align}
  \end{subequations}
The objective $f$ defined by \eqref{eq:local-optimization} is a polynomial and its gradient can easily be computed 
\begin{subequations}\label{eq:jacobian}
\begin{align}
  \frac{df}{dx_{k}^{(s)}}&=
    -\sum_{i=0}^{n-1}\left[ w_k\frac{dp_i({x}_{k})}{dx_{k}^{(s)}}
       \left(q(p_i)-\sum_{j=1}^M w_jp_i({x}_{j})\right)\right]\\
\frac{df}{dw_k}&=
      -\sum_{i=0}^{n-1}\left[ p_i({x}_{k})
       \left(q(p_i)-\sum_{j=1}^M w_jp_i({x}_{j})\right)\right],
\end{align}
\end{subequations}
for $s = 1, \ldots, d$ and $m = 1, \ldots, M$, and where $q(p_i) = \int p_i(x) \dx{\mu}(x)$.

We use a gradient-based nonlinear least squares method to solve the local optimization problem. Defining the optimization tolerance $\tau=10^{-10}$, the procedure exits when either $|f_{i}-{f_{i-1}}|<\tau f_i$ or $\lVert g_s\rVert_\infty < \tau$, where $f_i$ and $f_{i-1}$ are the values of the objective at steps $i$ and $i-1$ respectively, and $g_s$ is the value of the gradient scaled to respect the constraints \eqref{eq:local-constraints}.

This local optimization procedure in conjunction with the cluster
based initial guess can frequently find a quadrature rule of size $M$
as determined by the degree of freedom
heuristic~\eqref{eq:dof-heuristic}. However in some cases a quadrature
rule with $M$ points cannot be found for very high accuracy
requirements (in all experiments we say that a quadrature rule is found
found if the iterations yield $\lvert f_i\rvert<10^{-8}$). In these situations one might be able to find an $M$ point rule using another initial condition. (Recall the initial condition provided by $\ell_1$-minimization is found using a randomized candidate mesh.) However we found it more effective to simply increment the size of the desired quadrature rule to $M+1$. This can be done repeatedly until a quadrature rule with the desired accuracy is reached. While one fears that one may increment $M$ by a large number using this procedure, we found that no more than a total of $10$ of increments $(M \rightarrow M + 10)$ were ever needed. This is described in more detail in Section~\ref{sec:tp-measures}.

Note that both generating the initial condition using~\eqref{eq:bpdn-problem} and solving the local optimization problem~\eqref{eq:local-optimization} involve matching moments. We assume in this paper that moments are available and given; in our examples we compute these moments analytically (or to machine precision with high-order quadrature), unless otherwise specified.



\section{Numerical results}
\label{sec:numerical-results}
In this section we will explore the numerical properties of the multivariate quadrature algorithm we have proposed in Section \ref{sec:algorithm}. First we will numerically compare the performance of our algorithm with other popular quadrature strategies for tensor product measures. We will then demonstrate the flexibility our our approach for computing quadrature rules for non-tensor-product measures, for which many existing approaches are not applicable. Finally we will investigate the utility of our reduced quadrature rules for high-dimensional integration by leveraging selected moment matching and dimension reduction. 

Our tests will compare the method in this paper (Section \ref{sec:algorithm}), which we call \textsc{Reduced Quadrature}, to other standard quadrature methods. We summarize these methods below.
\begin{itemize}
  \item \textsc{Monte Carlo} -- The integration rule
    \begin{align*}
      \int_D f(x) \dx{\mu}(x) \approx \frac{1}{M} \sum_{m=1}^M f(X_m),
    \end{align*}
    where $X_m$ are independent and identically-distributed samples from the probability measure $\mu$.
  \item \textsc{Sparse grid} -- The integration rule for $\mu$ the uniform measure over $[-1,1]^d$ given by a multivariate sparse grid rule generated from a univariate Clenshaw-Curtis quadrature rule~\cite{Gerstner_G_NA_1998}. 
  \item \textsc{Cubature} -- Stroud cubature rules of degree 2, 3, and 5~\cite{Stroud_book_71}. These rules again integrate on $[-1,1]^d$ with respect to the uniform measure.
  \item \textsc{Sobol} -- A quasi-Monte Carlo Sobol sequence~\cite{Sobol_S_IJMPC_1995} for approximating integrals on $[-1,1]^d$ using the uniform measure.
  \item \textsc{Reduced quadrature} -- The method in this paper, described in Section \ref{sec:algorithm}.
  \item \textsc{$\ell_1$ quadrature} -- The ``initial guess" for the \textsc{Reduced quadrature} algorithm, using the solution of the LP algorithm summarized in Section \ref{sec:algorithm-initial}.
\end{itemize}
The \textsc{Sparse grid}, \textsc{Sobol}, and \textsc{cubature} methods are usually applicable only for integrating over $[-1,1]^d$ with the uniform measure. When $\mu$ has a density $w(x)$ with support $D \subseteq [-1,1]^d$, we will use these methods to evaluate $\int_D f(x) \dx{\mu}(x)$ by integrating $f(x) w(x)$ with respect to the uniform measure on $[-1,1]^d$, where we assign $w = 0$ on $[-1,1]^d \backslash D$.

\subsection{Tensor product measures}\label{sec:tp-measures}
Our reduced quadrature approach can generate quadrature rules for non-tensor-product measures but in this section we investigate the performance of our algorithm in the more standard setting of tensor-product measures. 

We begin by discussing the computational cost of computing our
quadrature rules. Specifically we compute quadrature rules for the
uniform probability measure on $D=[-1,1]^d$ in up to 10
dimensions for all total-degree spaces with degree at most 20 or with
subspace dimension at most 3003. In all cases we were able to generate an
efficient quadrature rule using our approach for which the number of
quadrature points was equal to or only slightly larger ($<10$ points)
than the number of points suggested by the
heuristic~\eqref{eq:dof-heuristic}. The number of each points in the
quadrature rule, the number of moments matched $\lvert\Lambda\rvert$,
the size of a quasi-optimal rule $L(\Lambda)$, the
number of points $\ceil{\lvert\Lambda\rvert/(d+1)}$ suggested by the heuristic ~\eqref{eq:dof-heuristic}, and the number of iterations taken by non-linear least squares algorithm, is presented in in Table~\ref{tab:uniform-quad-rules-max-degree}. 

The final two columns of Table \ref{tab:uniform-quad-rules-max-degree}, the number of points in the the reduced quadrature rules and the number of iterations used by the local optimization are given as ranges because the final result is sensitive to the random candidate sets used to generate the initial condition. The ranges presented in the table represent the minimum and maximum number of points and iterations used to generate the largest quadrature rules for each dimension for 10 different initial candidate meshes. We observe only very minor variation in the number of points, but more sensitivity in the number of iterations is observed. 

The number of iterations needed by the local optimization to compute the quadrature rules was always a reasonable number: at most a small multiple of the number of moments being matched. However, the largest rules did take several hours to compute on a single core machine due to the cost of running the optimizer (we used scipy.optimize.least\_squares) and forming the Jacobian matrix of the objective. 

Profiling the simulation revelated that long run times were due to the cost of evaluating the Jacobian \eqref{eq:jacobian} of the objective function, and the singular value decomposition repeatedly called by the optimizer. This run-time could probably be reduced by using a Krylov-based nonlinear least squares algorithm, and by computing the Jacobian in parallel. We expect that such improvements would allow one to construct such rules in higher dimensions in a reasonable amount of time; however, even our unoptimized code was able to compute such rules on a single core in a moderate number of dimensions.

\begin{table}[ht]
\begin{center}
\begin{tikzpicture}
 \node (tbl) {
  \begin{tabular}{ c  c  c  c  c  c  c}
    \arrayrulecolor{darkblue}
    Dimension & Degree&  $\lvert\Lambda\rvert$ & $L(\Lambda)$ &
                                                               $\ceil{\lvert\Lambda\rvert/(d+1)}$ & No. Points &No. Iterations\\[1.ex]
    2 & 20 & 231 & 66 & 77 &77-79 & 262-854 \\                     \midrule
    3 & 20 & 1771 & 286 & 443 &445-447 & 736-1459\\                \midrule
    4 & 13 & 2380 & 210 & 476 &479-480 & 782-2033\\                \midrule
    5 & 10 & 3003 & 252 & 501 &506-508 & 3181-4148\\               \midrule
    10 & 5 & 3003 & 66 & 273 &273-274 & 511-2229 \\[0.5ex]      
  \end{tabular}
};
\begin{pgfonlayer}{background}
\draw[rounded corners,top color=lightblue,bottom color=darkblue!50!black,
    draw=white] ($(tbl.north west)+(0.14,0)$)
    rectangle ($(tbl.north east)-(0.13,0.9)$);
\draw[rounded corners,top color=white,bottom color=darkblue!50!black,
    middle color=lightblue,draw=blue!20] ($(tbl.south west)
    +(0.12,0.5)$) rectangle ($(tbl.south east)-(0.12,0)$);
\draw[top color=blue!1,bottom color=gray!50,draw=white]
    ($(tbl.north east)-(0.13,0.6)$)
    rectangle ($(tbl.south west)+(0.13,0.2)$);
\end{pgfonlayer}
\end{tikzpicture}
\end{center}
\caption{Results for computing \textsc{Reduced quadrature} rules for total-degree spaces for the uniform measure on $[-1,1]^d$. Tabulated are the number of moments matched ($|\Lambda|$), the theoretical lower bound on the quadrature rule size from \eqref{eq:L-definition}, the number of the quadrature points given by the counting heuristic \eqref{eq:dof-heuristic}, the number of {reduced quadrature} points found, and the number of iterations required in the optimization. The {reduced quadrature} algorithm output is random since the initial candidate grid is a random set. Thus, the final two columns give a range of results over 10 runs of the algorithm.}
   \label{tab:uniform-quad-rules-max-degree}
\end{table}

To compare the performance of our reduced quadrature rules to existing algorithms consider the the corner-peak function often used to test quadrature methods~\cite{Genz_Book_1987},
\begin{align}\label{eq:genz-cp}  
 f_{\mathrm{CP}}(\V{x})=\left(1+\sum_{i=1}^d c_i\, x_i \right)^{-(d+1)},\quad \V{x}\in\Gamma=[0,1]^d
\end{align}
The coefficients $c_i$ can be used to control the variability over $D$ and the effective dimensionality of this function. We generate each $c_i$ randomly from the interval $[0,1]$ and subsequently normalize them so that $\sum_{i=1}^d c_i = 1$. For $X$ a uniform random variable on $D = [-1,1]^d$, the mean and variance of $f(X)$ can be computed analytically; these values correspond to computing integrals over $D = [-1,1]^d$ with $\mu$ the uniform measure.

Figure ~\ref{fig:genz-cp-convergence} plots the convergence of the
error in the mean of the corner-peak function using the reduced
quadrature rules generated for $\mu$ the uniform probability measure
on $[-1,1]^d$ with $d=2,3,4,5,10$. Because generating the initial
condition requires randomly sampling over the domain of integration
for a given degree,  we generate 10 quadrature rules and plot both the
median error and the minimum and maximum errors. There is sensitivity to the random candidate set used to generate the initial condition, however for each of the 10 repetitions a quadrature rule was always found and the error in the approximation of the mean of $f$ using these rules converges exponentially fast. 

For a given number of samples the error in the reduced quadrature rule estimate of the mean is significantly lower than the error of a Clenshaw-Curtis-based sparse grid. It is also more accurate than Sobol sequences up to 10 dimensions. Furthermore the reduced quadrature rule is competitive with the Stroud degree  2,3 and 5 cubature rules. However unlike the Stroud rules the polynomial exactness of the reduced quadrature rule is not restricted to low degrees (shown here) nor is it even resticted to total-degree spaces (as discussed in Section~\ref{sec:poly-spaces}).

In Figure ~\ref{fig:genz-cp-convergence} we also plot the error in the quadrature rule used as the initial condition for the local optimization used to generated the reduced rule. As expected, the error of the initial condition is larger than the final reduced rule for a given number of points. Moreover it becomes increasingly difficult to find an accurate initial condition as the degree increases. In two-dimensions, as the degree is increased past 13, the error in the initial condition begins to increase. Similar degradation in accuracy also occurs in the other higher dimensional examples. However, even when the accuracy of the initial condition still is extremely low, it still provides a good starting point for the local optimization, allowing us to accurately generate the reduced quadrature rule. 

\begin{figure}
\begin{center}
\includegraphics[width=0.48\textwidth]{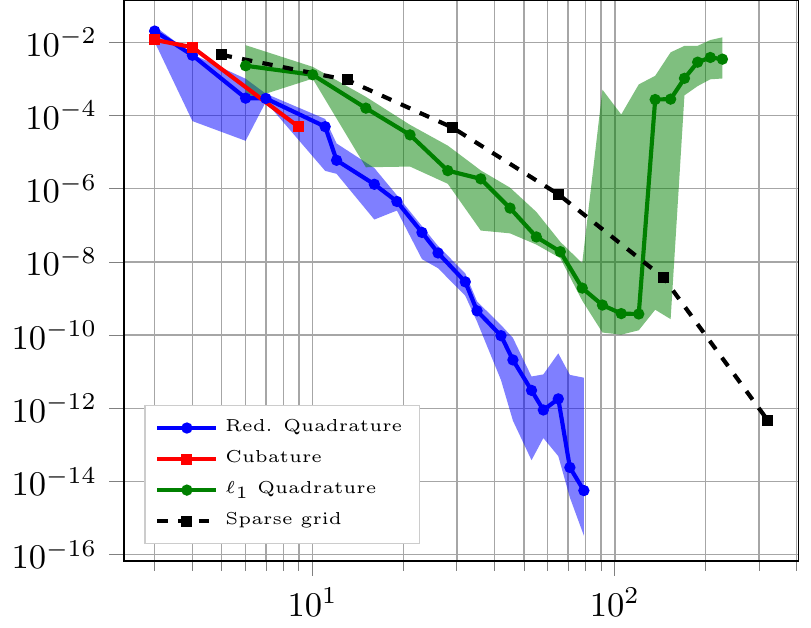}
\includegraphics[width=0.48\textwidth]{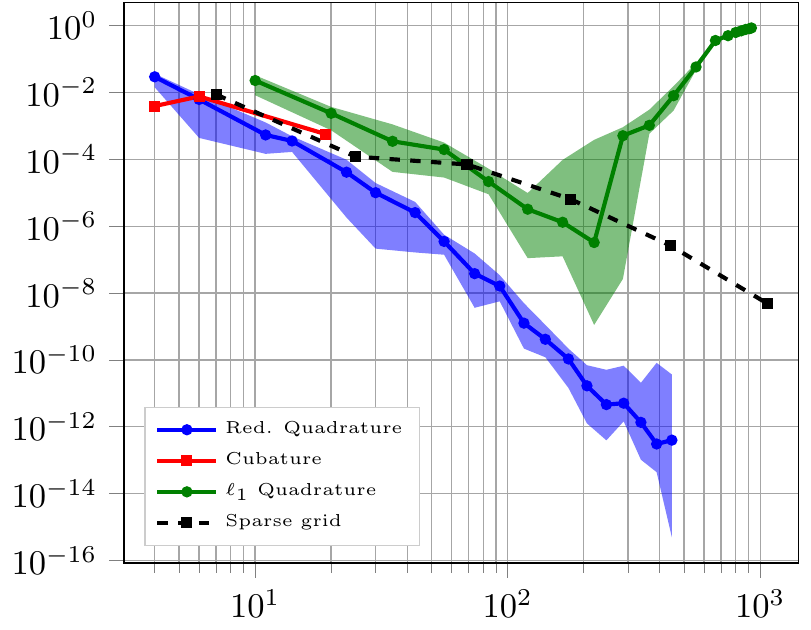}\\
\includegraphics[width=0.48\textwidth]{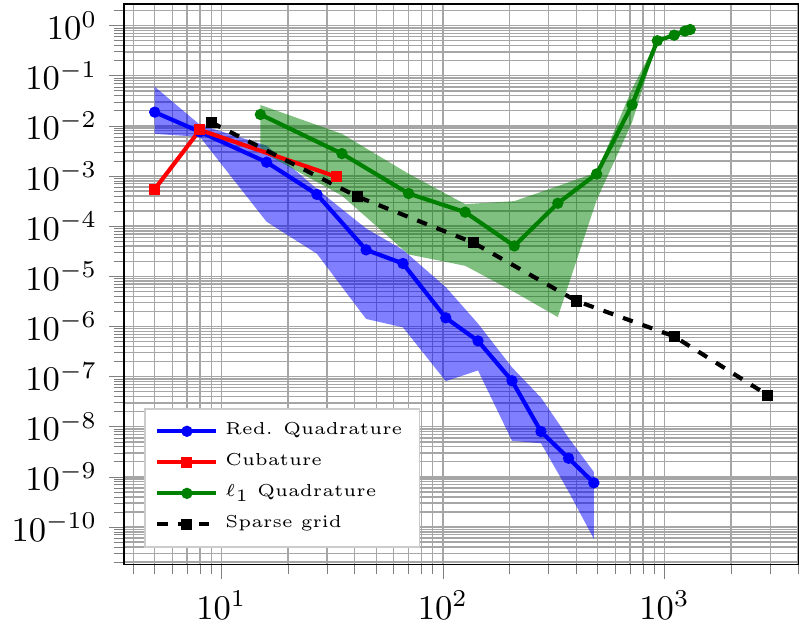}
\includegraphics[width=0.48\textwidth]{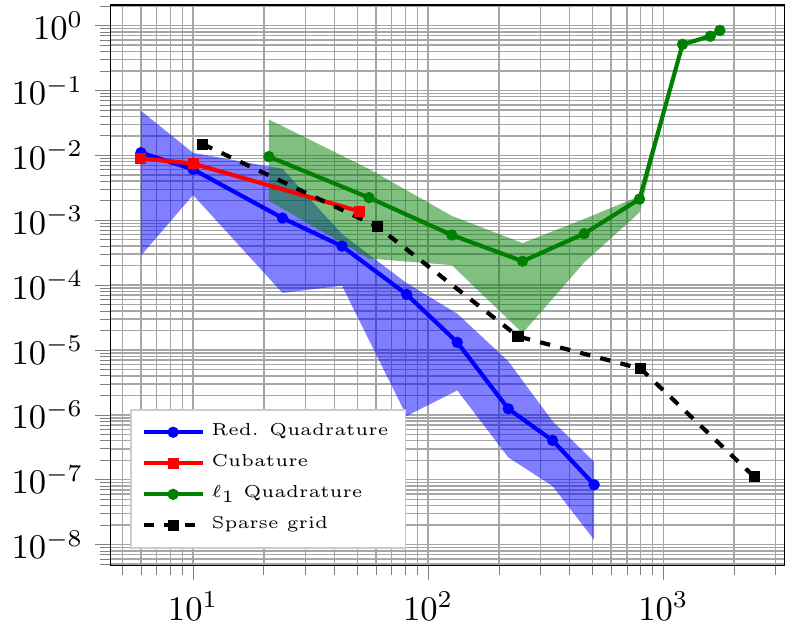}\\
\includegraphics[width=0.48\textwidth]{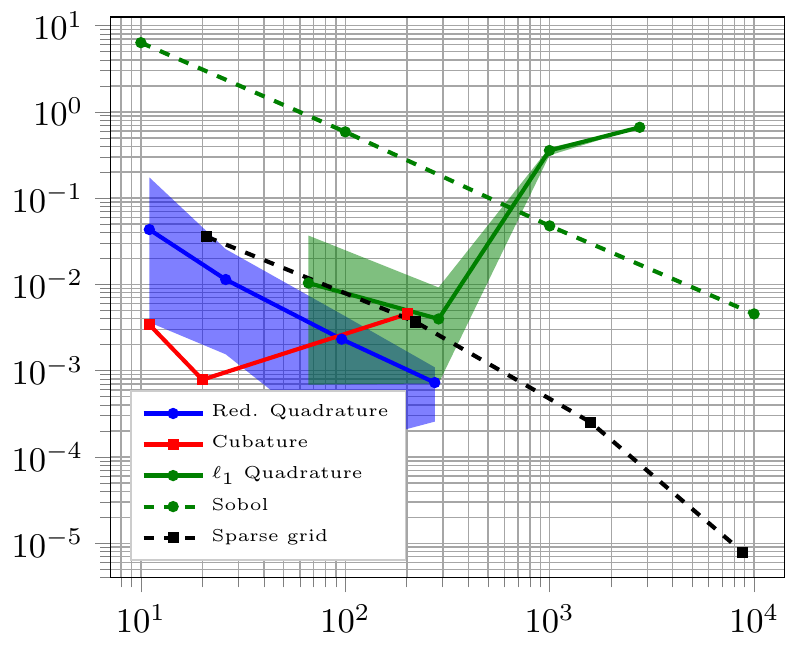}\\
\end{center}
\caption{Convergence of the error in the mean of the corner-peak
  function $f_\mathrm{CP}$~\eqref{eq:genz-modified-cp} computed using
  the reduced quadrature rule generated for the uniform
  measure. Convergence is shown for $d=2,3,4,5,10$ which are
  respectively shown left to right starting from the top left. Solid
  lines represent the median error of the total-degree quadrature
  rules for 10 different initial conditions. Transparent bands
  represent the minimum and maximum errors of the same 10 repetitions.}
\label{fig:genz-cp-convergence}
\end{figure}



\subsection{Beyond tensor product measures}\label{sec:non-tp-measures}
Numerous methods have been developed for multivariate quadrature for tensor-product measures, however much less attention has been given to computing quadrature rules for arbitrary measures, for example those that exhibit non-linear correlation between dimensions. In this setting Monte Carlo quadrature is the most popular method of choice. 

\subsubsection{Chemical reaction model}
The reduced quadrature method we have developed can generate quadrature rules to integrate over arbitrary measures. Consider the following model of competing species absorbing onto a surface out of a gas phase
\begin{align}\label{eq:chemical-species}
\begin{split}
\frac{du_1}{dt} &= az-cx_1 - 4du_1u_2\\
\frac{du_2}{dt} &= 2x_2z^2 - 4du_1u_2\\
\frac{du_3}{dt} &= ez - fu_3\\
z=u_1&+u_2+u_3,\quad u_1(0)=u_2(0)=u_3(0)
\end{split}
\end{align}
The constants $c$, $d$, $e$, and $f$ are fixed at the nominal values $c=0.04$, $d=1.0$, $e=0.36$, and $f=0.016$. The parameters $x_1$ and $x_2$ will vary over a domain $D$ endowed with a non-tensor-product probability measure $\mu$. Viewing $X = (X_1, X_2) \in \R^2$ as a random variable with probability density $\dx{\mu}(x)$, we are interested in computing the mean of the mass fraction of the third species $u_3(t=100)$ at $t=100$ seconds. 

We will consider two test problems, problem $\mathrm{I}$ and problem $\mathrm{II}$, each defined by its own domain $D$ and measure $\mu$. We therefore have two rectangular domains $D_{\mathrm{I}}$ and $D_{\mathrm{II}}$ with measures $\mu_{\mathrm{I}}$ and $\mu_{\mathrm{II}}$, respectively. The domains and the measures are defined via affine mapping of a canonical domain and measure:
\begin{align}\label{eq:banana-density}
  \dx{\mu}(x) &= C\exp(-(\frac{1}{10}x_1^4 + \frac{1}{2}(2x_2-x_1^2)^2)), & {x}&\in D = [-3,3]\times[-2,6],
\end{align}
where $C$ is a constant chosen to normalize $\mu$ as a probability measure. This density is called a ``banana density", and is a truncated non-linear transformation of a bivariate standard Normal Gaussian distribution.  We define $(D_{\mathrm{I}}, \mu_{\mathrm{I}})$ and $(D_{\mathrm{II}}, \mu_{\mathrm{II}})$ as the result of affinely mapping $D$ to the domains
\begin{align*}
  D_{\mathrm{I}}&= [0,4.5]\times[5,35], & D_{\mathrm{II}} &= [1.28,1.92]\times[16.6,24.9].
\end{align*}

The response surface of the mass fraction over the integration domain of problem I is shown in the right of Figure~\ref{fig:banana-density}. The response has a strong non-linearity which makes it ideal for testing high-order polynomial quadrature. For comparison the integration domain of problem II is also depicted in the right of Figure~\ref{fig:banana-density}. The domain of problem II is smaller and thus the non-linearity of the response is weaker, consequently integrating the mean of problem II is easier than integrating the mean of problem II  for polynomial quadrature methods. 
\begin{figure}
\begin{center}
\includegraphics[width=0.4\textwidth]{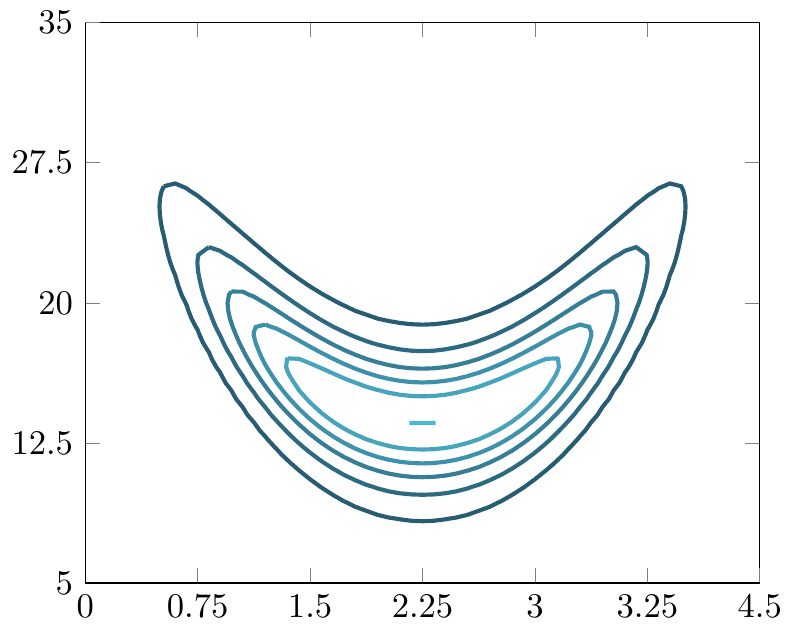}
\includegraphics[width=0.49\textwidth]{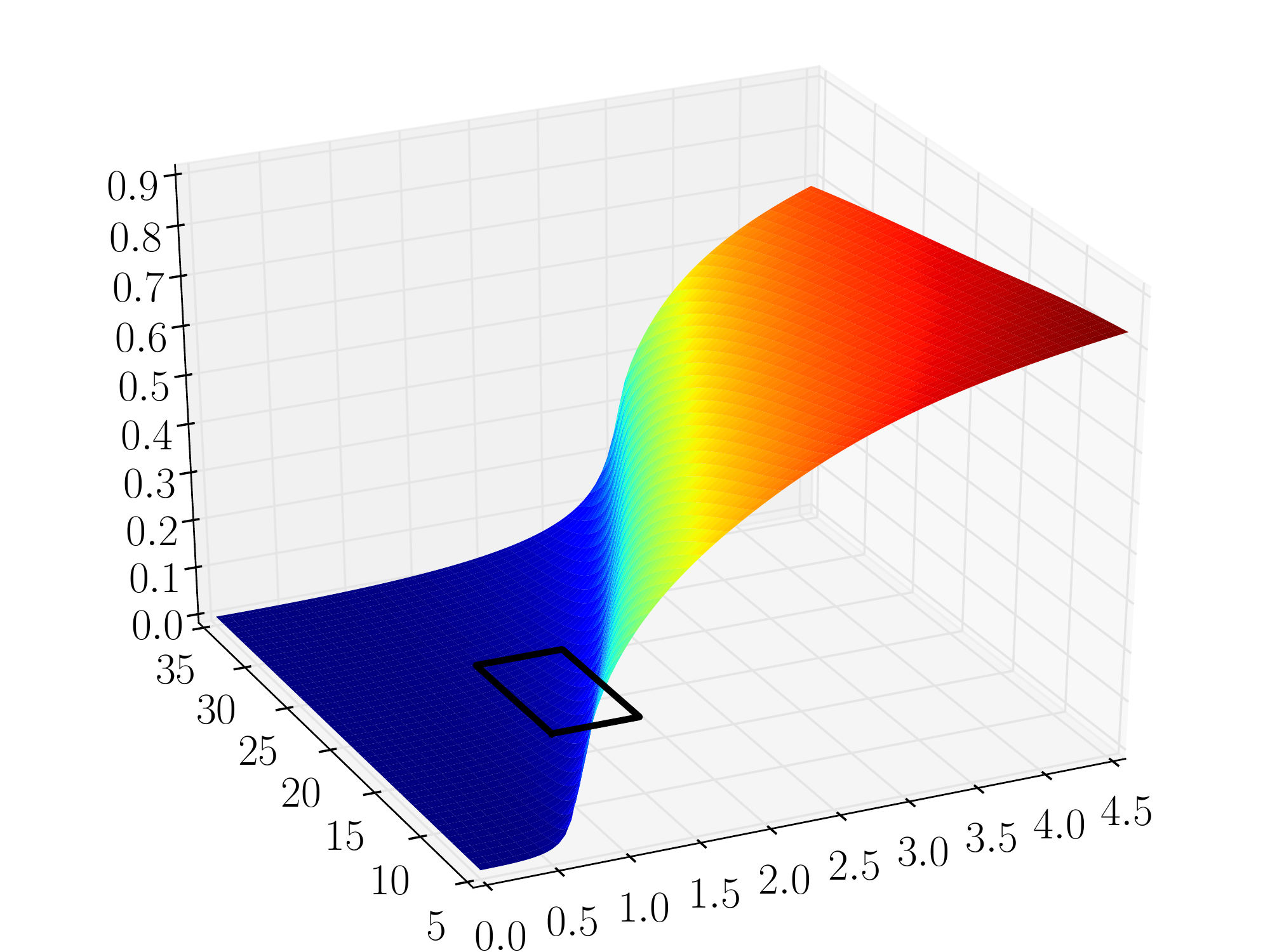}
\end{center}
\caption{(Left) Contour plot of the banana density \eqref{eq:banana-density} mapped to $D_I$. (Right) Response surface of the mass fraction $u_3$ at $t=100$ predicted by the chemical reaction model. The response is plotted over the entire domain of problem I and the black box represents the smaller domain of problem II.
}
\label{fig:banana-density}
\end{figure}

Figure~\ref{fig:banana-chemical-reaction-convergence} compares the
convergence of the of the error in the mean of the mass fraction of
the chemical reaction model computed using the reduced quadrature
rule, with the estimates of the mean computed using Monte Carlo
sampling and Clenshaw-Curtis based sparse grids. Unlike previous
examples the mean cannot be computed analytically so instead we
compute the mean using $10^6$ samples from a $2$-dimensional Sobol
sequence. Sparse grids can only be constructed for tensor-product
measures, so here we investigate performance of sparse grids by
including the probability density in the integrand and integrating
with respect to the uniform measure. This is the most common strategy to tackle a non-tensor-product integration problem using a tensor-product quadrature rule. 

For the more
challenging integral defined over $D_\mathrm{I}$, the reduced
quadrature error out-performs Monte Carlo and sparse grid quadrature
however the difference is more pronounced when integrating over
$D_\mathrm{II}$. The apparent slower rate of convergence of
reduced quadrature for problem I is because a high-polynomial degree
is needed to accurately approximate the steep response surface features over this
domain. The performance of the reduced quadrature method is related to
how well the integrand can be approximated by a polynomial.   
\begin{figure}
\begin{center}
\includegraphics[width=0.49\textwidth]{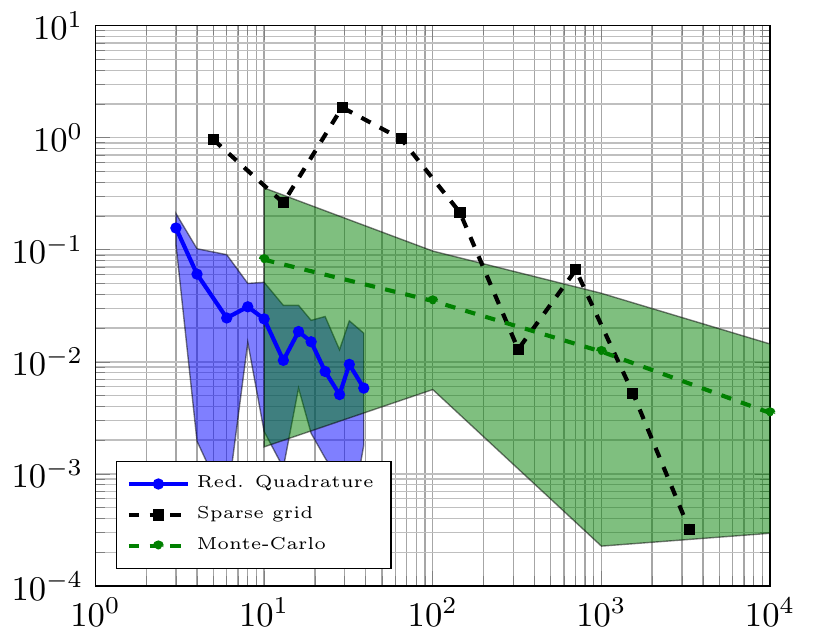}
\includegraphics[width=0.49\textwidth]{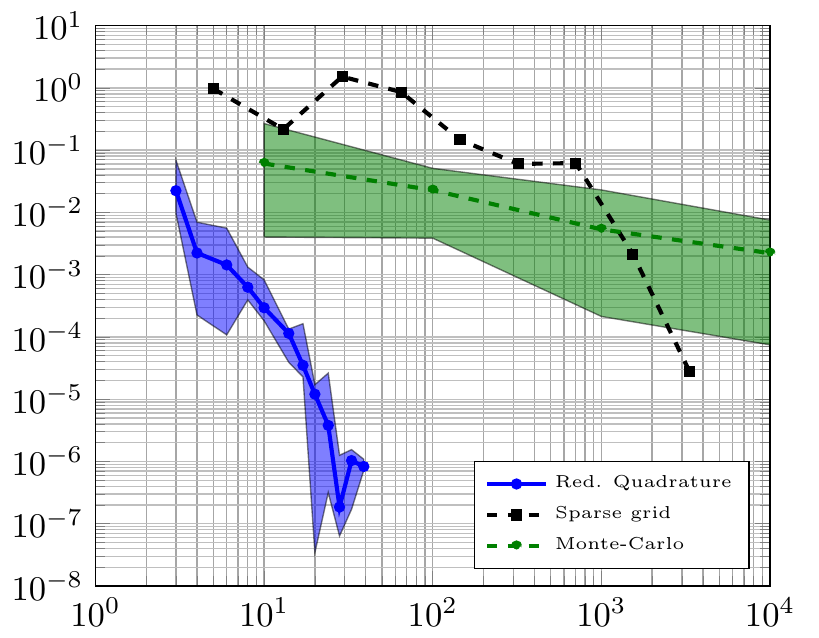}
\end{center}
\caption{Convergence of the error in the mean of the mass fraction of the chemical reaction model computed using the reduced quadrature rule generated for the banana-density over (right) $D_\mathrm{II}$ and (left) $D_\mathrm{I}$. Error is compared to popular alternative quadrature methods.}
\label{fig:banana-chemical-reaction-convergence}
\end{figure}

\subsubsection{Sample-based moments}\label{sec:non-tp-sample-moments}
Note that both generating the initial condition~\eqref{eq:bpdn-problem} using and solving the local optimization problem~\eqref{eq:local-optimization} involve matching moments. Throughout this paper we have computed the moments of the polynomials we are matching analytically (or to machine precision with high-order quadrature). However situations may arise when one only has samples from the probability density of a set of random variables. For example Bayesian inference~\cite{Stuart_AN_2010} is often used to infer densities of random variables conditional on available observational densities. The resulting so called posterior densities are almost never tensor-product densities. Moreover it is difficult to compute analytical expressions for the posterior density and so Markov Chain Monte Carlo (MCMC) sampling is often used to draw a set of random samples from the posterior density. 

Consider a model $f_\text{p}({x}):\realsmap{d}{1}$ parameterized by $d$ variables ${x}=(x_1,\ldots,x_{d})\in D \subset\reals^{d}$,
predicting an unobservable quantity of interest (QoI). There is a true underlying value of $x$ that is unknown. In the example \eqref{eq:chemical-species}, $f_p$ is the mass fraction $u_3$ of the chemical reaction model. We wish to quantify the effect of the uncertain variables on the model prediction, and we use Bayesian inference to accomplish this.

In the standard inverse problems setup, we have no direct measurements of $f_p$ but we can make observations ${{y}_o}$ of other quantities which we can use to constrain estimates of uncertainty in the unobservable QoI. To make this precise, let $f_\text{o}({x}) : \realsmap{d}{n_o}$ be a model, parameterized by the same $d$ random variables ${x}$, which predicts a set of $n_o$ observable quantities. Bayes rule can be used to define the posterior density for the model parameters ${x}$ given observational data ${{y}_o}$:
\begin{align}\label{eq:posterior}
\pi({x}|{{y}_o})=\frac{\pi({{y}_o}|{x})\pi({x})}{\int_{\Gamma}
({{y}_o}|{x})\pi({x})d{x}},
\end{align}
where any prior knowledge on the model parameters is captured 
through the prior density $\pi(\V{x})$.

The construction of the posterior is often not the end goal of an analysis instead one is often interested in statistics on the unobservable QoI. Here we will
focus on moments of the data informed predictive distribution, for example the mean
prediction
\begin{equation}\label{eq:post-pred-moments}
m_p = \int_D f_\text{p}({x})\pi({x}|{{y}_o})\dx{{x}}
\end{equation}

In practice the posterior of the chemical model parameters may be obtained by observational data ${y}$ from chemical reaction experiments and an observational model $f_\text{o}$ that is used to numerically simulate those experiments. However for ease of discussion we will assume the posterior distribution~\eqref{eq:posterior} is given by the banana-type density defined over $D_\mathrm{I}$ or $D_\mathrm{II}$.  The banana density \eqref{eq:banana-density} has been used previously to test Markov Chain Monte Carlo (MCMC) methods~\cite{Parno_M_arxiv_2014} and thus is a good test case for Bayesian inference problems. 

As in the previous section we will use the reduced quadrature method
to compute $m_p$, however here we will investigate the performance
of the quadrature rules that are generated from \textit{approximate} moments. The approximate moments are those computed via Monte Carlo sampling from $\pi(x | y_o)$.  These approximate moments are used to generate reduced quadrature rules, i.e. they are used as inputs when solving~\eqref{eq:bpdn-problem} and~\eqref{eq:local-optimization}. Figure~\ref{fig:banana-chemical-reaction-convergence-samples} illustrates the effect of using a finite number of samples to approximate the polynomial moments. The left of the figure depicts the convergence of the error in the mean of the mass fraction for the banana density defined over $D_\mathrm{I}$. The right shows errors for the banana density defined over $D_\mathrm{II}$. The right-hand figure illustrates that the accuracy of the quadrature rule is limited by the accuracy of the moments used to generate the quadrature rule. Once the error in the approximation of the mean reaches the accuracy of the moments, the error stops decreasing when the number of quadrature points is increased. The error saturation point can be roughly estimated by the typical Monte Carlo error which scales like $P^{-1/2}$, where $P$ is the number of samples used to estimate the polynomial moments. The saturation of error present in the right-hand figure is not as evident in the left-hand figure, and this is because the error of the quadrature rules using exact moments is greater than the error in the Monte Carlo estimate of the moments using based upon $10^4$ and $10^6$ samples.  

Our Monte Carlo samples from $\pi(x|y_o)$ were generated using rejection sampling, which is an exact sampler. While MCMC is the tool of choice for sampling from high-dimensional non-standard distributions, it is an approximate sampler. If we had generated samples using MCMC, then our approximate moments contain two error sources: that from finite sample size, and that from approximate sampling. For this reason, we opted to use the exact rejection sampling technique. However, we expect that the results in Figure \ref{fig:banana-chemical-reaction-convergence-samples} would look similar when using MCMC samples.

\begin{figure}
\begin{center}
\includegraphics[width=0.49\textwidth]{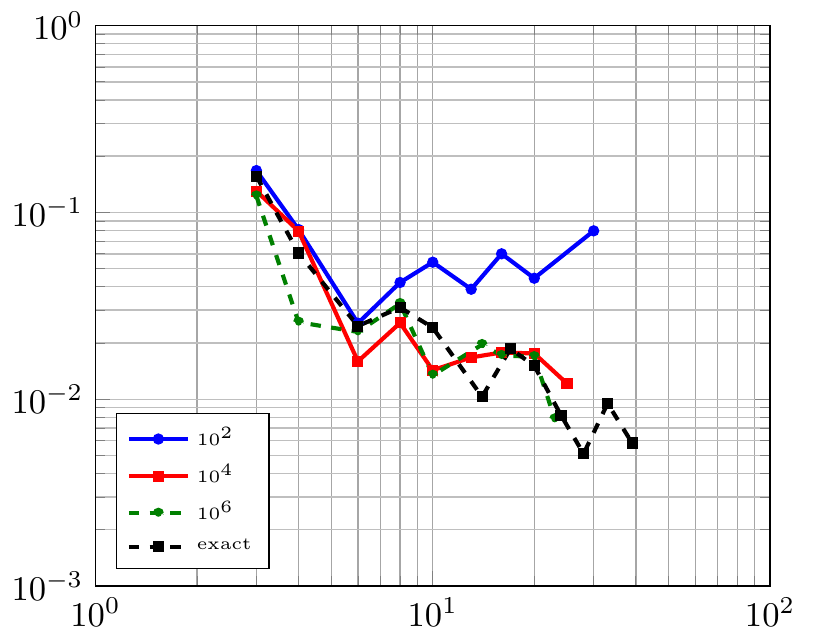}
\includegraphics[width=0.49\textwidth]{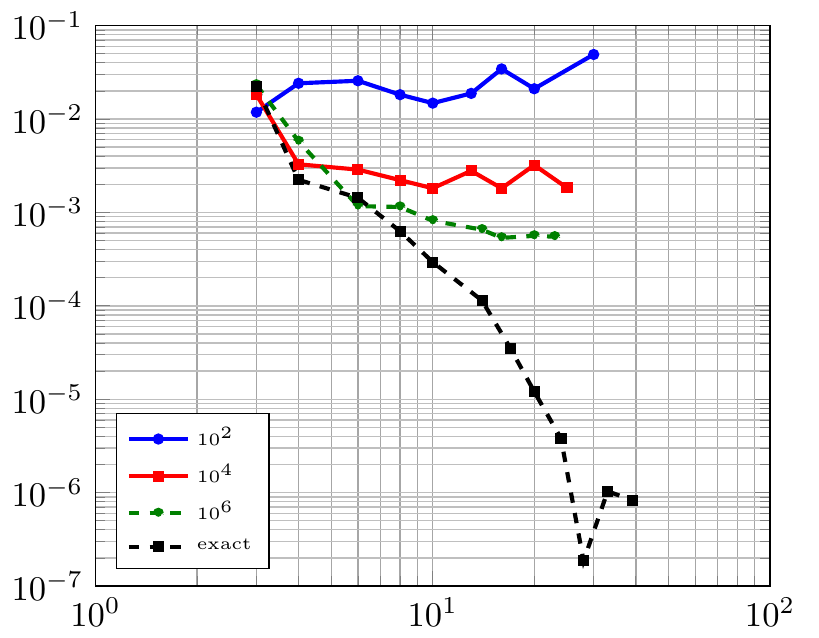}
\end{center}
\caption{Convergence of the error in the mean of the mass fraction of the chemical reaction model computed using the reduced quadrature rule generated for the banana-density using approximate moments. The $x$ axis in the plots indicates the number of points in a reduced quadrature rule generated from approximate moments. The numbers in the legend indicate the number of Monte Carlo samples used to approximate moments.
(Right) Convergence over $D_\mathrm{I}$ and (Left) $D_\mathrm{I}$.} 
\label{fig:banana-chemical-reaction-convergence-samples}
\end{figure}

\subsection{High-dimensional quadrature}
The cost of computing integrals with total-degree quadrature rules increases exponentially with dimension. This cost can be ameliorated if one is willing to consider subspaces that are more exotic than total-degree spaces. In this section we show how our reduced quadrature rules can generate quadrature for non-standard subspaces, taking advantage of certain structures that may exist in an integrand. Specifically we will demonstrate the utility of reduced quadrature rules for integrating high-dimensional functions that can be approximated by (i)  low-order ANOVA expansions and (ii) by a function of a small number of linear combinations of the function variables (such functions are often referred to as ridge functions). 

\subsubsection{ANOVA decompositions}\label{sec:poly-spaces}
Dimension-wise decompositions of the input-output relationship have arisen as a
successful means of delaying or even breaking the curse of dimensionality for
certain applications~\cite{Griebel_H_JC_2010,Foo_K_JCP_2010,Ma_Z_JCP_2009,Jakeman_R_SGA_2013}. Such an approach represents the model outputs as a 
high-dimensional function $f(x)$ dependent on the model inputs
$x=(x_1,\ldots,x_d)$. It uses an ANOVA type decomposition
\begin{equation}
f(x_1,\ldots,x_d)=f_0+\sum_{n=1}^df_i(x_i)+\sum_{i,j=1}^df_{i,j}(x_i,
x_j)+\cdots+f_{i,\ldots,d}(x_i,\ldots,x_d)
\end{equation}
which separates the function into a sum of subproblems. The first term is the
zero-th order effect which is a constant throughout the d-dimensional variable
space. The $f_i(x_i)$ terms are the first-order terms which represent the
effect of each variable acting independently of all others. The
$f_{i,j}(x_i,x_j)$ terms are the second-order terms which are the contributions
of $x_i$ and $x_j$ acting together, and so on.
In practice only a small number of interactions contribute significantly to the
system response and consequently only a low-order approximation is needed to
represent the input-output mapping accurately~\cite{Wang_S_SISC_2005}. 

In this section we show that reduced quadrature rules can be constructed for functions admitting ANOVA type structure. Specifically consider the following modified version of the corner peak function~\eqref{eq:genz-cp}  
\begin{align}\label{eq:genz-modified-cp}  f_{\mathrm{MCP}}=\sum_{i=1}^{d-1}\left(1+c_ix_i+c_{i+1}x_{i+1}\right)^{-3}, \quad x\in D = [0,1]^d
\end{align}
with $d=20$.
This function has at most second order ANOVA terms, and can be integrated exactly using the uniform probability measure on $D$. 

We can create a quadrature rule ideally suited to integrating functions that have at most second-order ANOVA terms. Specifically we need only customize the index set $\Lambda$ that is input to Algorithm~\ref{alg:reduced-quadrature}. To emphasize a second-order ANOVA approximation, we compute moments of the form
$$\int_\Gamma p_\alpha(x)\,d\mu(x),\quad \forall\alpha\in\Lambda=\{\alpha\;|\;\lVert\alpha\rVert_0\le 2 \textrm{ and } \lVert\alpha\rVert_1\le k\}$$
for some degree $k$.

In Figure ~\ref{fig:genz-modified-cp-error} we compare the error in the mean computed using our reduced quadrature method, with the errors in the estimates of the mean computed using popular high-dimensional integration methods, specifically Clenshaw Curtis sparse grids, Quasi Monte Carlo integration based upon Sobol sequences and low-degree (Stroud) cubature rules. By tailoring the reduced quadrature rule to the low-order ANOVA structure of the integrand, the error in the estimate of the mean is orders of magnitude smaller than the estimates computed using the alternative methods. 
\begin{figure}
\begin{center}
\raisebox{.09\textwidth}{
\includegraphics[width=0.24\textwidth]{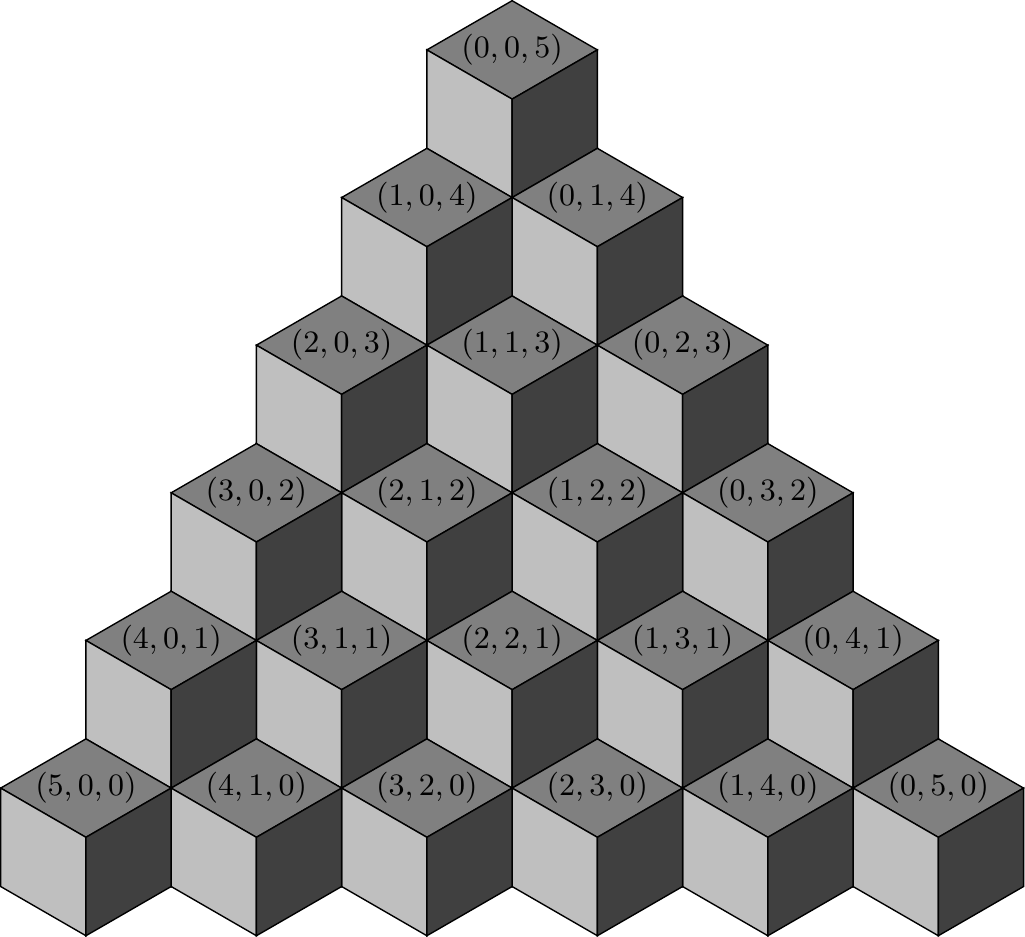}
\includegraphics[width=0.24\textwidth]{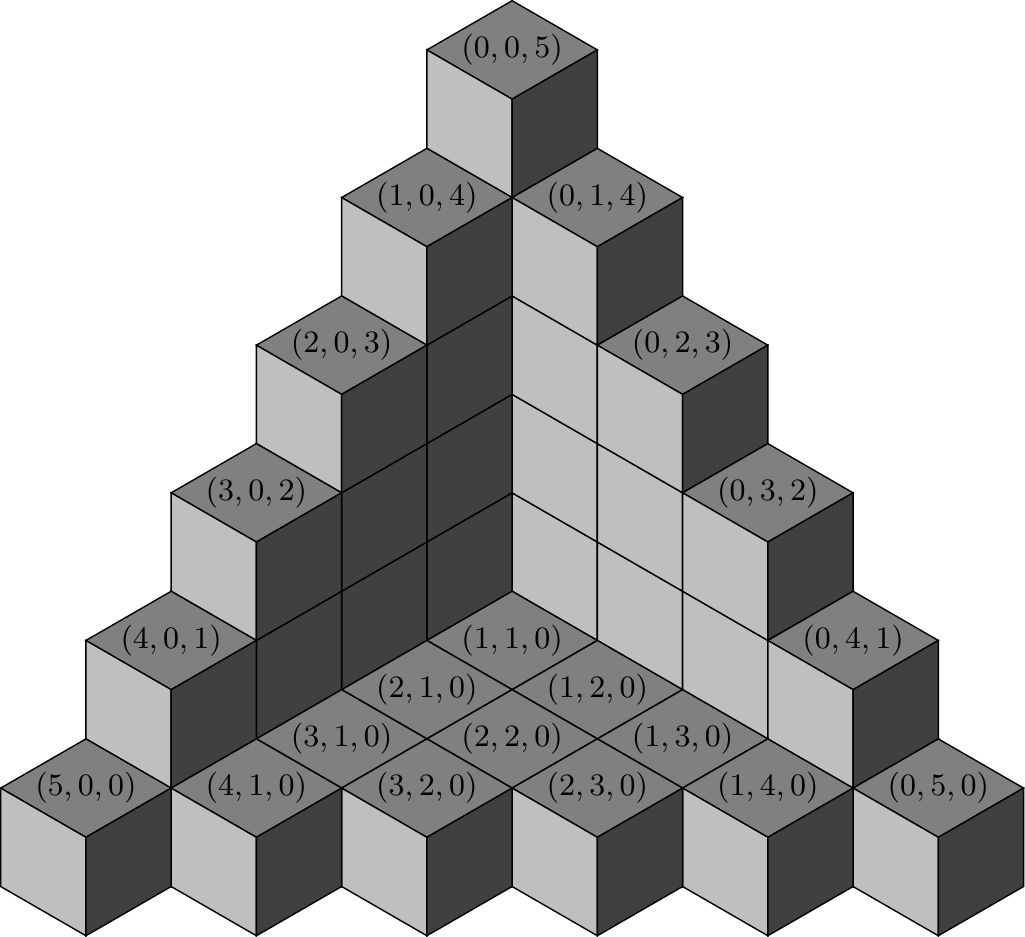}
}
\includegraphics[width=0.49\textwidth]{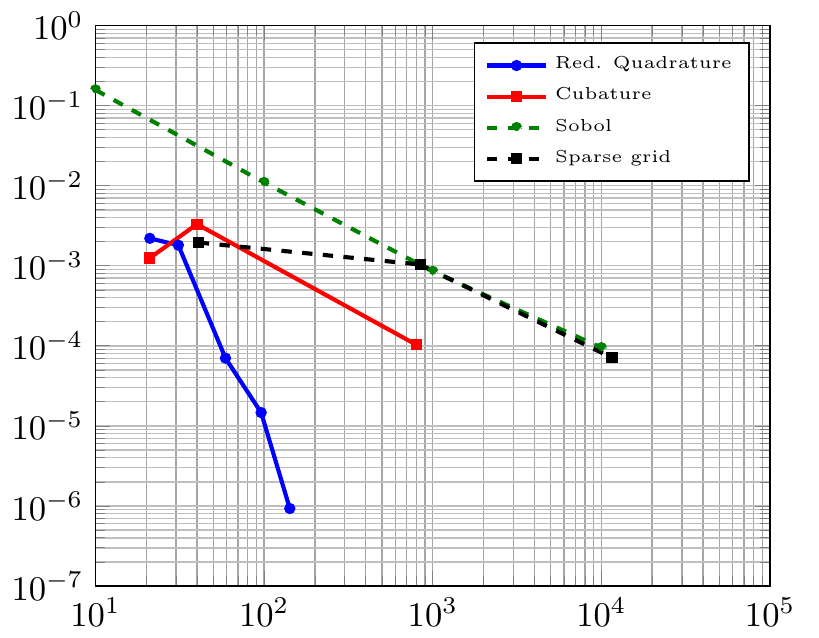}
\end{center}
\caption{(Left) Comparison of a three dimensional total degree index set of degree 5 with a 2nd-order ANOVA index set of degree 5. (Right) Convergence of the error in the mean of the modified corner-peak function $f_\mathrm{MCP}$~\eqref{eq:genz-modified-cp} computed using reduced quadrature rules for the uniform measure.}
\label{fig:genz-modified-cp-error}
\end{figure}
Note that sparse grids are a form of an anchored ANOVA expansion~\cite{Griebel_H_JC_2010} and delay the curse of dimensionality by assigning decreasing number of samples to resolving higher order ANOVA terms. However, unlike our reduced quadrature rule method sparse grids cannot be tailored to the exact ANOVA structure of the function.

\subsubsection{Ridge functions}\label{sec:dim-reduction}
In this section we show that our algorithm can be used to integrate high-dimensional ridge functions.
Ridge functions are multivariate functions that can be expressed as a function of a small number of linear combinations of the input variables variables \cite{Pinkus_book_2015}. We define a ridge function to be a function of the form 
$f: \mathbb{R}^d\rightarrow\mathbb{R}$ that can be expressed as a function $g$ of $s < d$ rotated variables, 
\begin{align*}
  f(z) &= g(\rotation z), & \rotation &\in\mathbb{R}^{s\times d}.
\end{align*}
In applications it is common for $d$ to be very large, but for $f$ to be a(n approximate) ridge function with $s \ll d$. When integrating a ridge function one need not integrate in $\mathbb{R}^d$ but rather can focus on the more tractable problem of integrating in $\mathbb{R}^s$. The difficulty then becomes integrating with respect to the transformed measure of the high-dimensional integral in the lower-dimensional space, which is typically unknown and challenging to compute.

When $d$-dimensional space is a hypercube, then the corresponding $s$-dimensional domain of integration is a multivariate zonotope, i.e., a convex, centrally symmetric polytope that is the $s$-dimensional linear projection of a $d$-dimensional hypercube. The vertices of the zonotope are a subset of the vertices of the $d$-dimensional hypercube projected onto the $s$-dimensional space via the matrix $\rotation $. 

When computing moment-matching quadrature rules on zonotopes we must amend the local optimization problem to include linear inequality constraints, to enforce that the quadrature points chosen remain inside the zonotope. The inequality constraints of the zonotope are the same constraints that define the convex hull of the zonotope vertices.\footnote{Note that most non-linear least squares optimizers do not allow the specification of inequality constraints so to compute quadrature rules on a zonotope we we used a sequential quadratic program.} Computing all the vertices of the zonotope $\{\rotation v\;|\;v\in[-1,1]^d\}$ can be challenging: The number of vertices of the zonotope grows exponentially with the large dimension $d$. To address this issue we use a randomized algorithm~\cite{Stinson_GC_ARXIV_2016} to find a subset of vertices of the zonotope.  This algorithm produces a convex hull that is a good approximation of the true zonotope with high-probability. In all our testing we found that the approximation of the zonotope hull did not noticeably affect the accuracy of the quadrature rules we generated.

We assume that $y \in \R^d$ is the $d$-dimensional variable with a measure $\nu$. Since $d$ is large, $\nu$ is typically a tensor-product measure. In the projected space $x \coloneqq A y \in \R^s$, this induces a new measure $\mu$ on the zonotope $D$ that is not of tensor-product form. With this setup, the $\mu$-moments can be computed analytically by taking advantage of the relationship $x=\rotation z$. For further details see Appendix~\ref{app:ridge-function-moments}.

Once a quadrature rule on a zonotope $D$ is constructed, some further work is needed before it can be applied to integrate the function $f$ on the original $d$-dimensional hypercube. Specifically we must transform the quadrature points $x\in D$ back into the hypercube. The inverse transformation of $\rotation $ is not unique, but if the function is (approximately) constant in the directions orthogonal to $x$, then any choice will do. In our example we set this transformation as $y = \rotation ^T x\in D$. The integral of the ridge function can then be approximated by
$$
\int_{D_y} f(y)\,d\nu(x) \approx \sum_{i=1}^M f(\rotation ^T x_i) w_i
$$   
where $(x_i, w_i)$ are a quadrature rule generated to integrate over the $s$-dimensional domain $D$ with the non-tensor-product measure $\mu$.

In the following we will consider the integration of a high-dimensional ridge function with $\nu$ the uniform probability measure on $[-1,1]^d$. We will again consider the integrating the moments mass fraction of the third species $u_3(t=100)$ of the competing species model from Section~\ref{sec:non-tp-measures}. However now we set $x=\rotation y$, where $y$ with $d=20$ and $\rotation \in\mathbb{R}^{2\times 20}$ is a randomly generated matrix with orthogonal rows. This makes the mass fraction a ridge function of two variables.

For a realization of $\rotation $ we plot the two resulting dimensional zonotope that defines the domain of integration of the variables $x$ in Figure~\ref{fig:zonotope} (left). The new probability density $\mu$ is depicted in the same figure. It is obvious that the transformed density $\mu$ is no longer uniform. In Figure ~\ref{fig:zonotope} (right) we plot the convergence of the error in the mean value of the ridge function computed using reduced quadrature rules. The error for the reduced quadrature approach decays exponentially fast and for a given number of function evaluations is orders of magnitudes smaller than the error obtained using Clenshaw-Curtis sparse grids and Sobol sequences (a QMC method) in the 20 dimensional space.
\begin{figure}
\begin{center}
\includegraphics[width=0.45\textwidth]{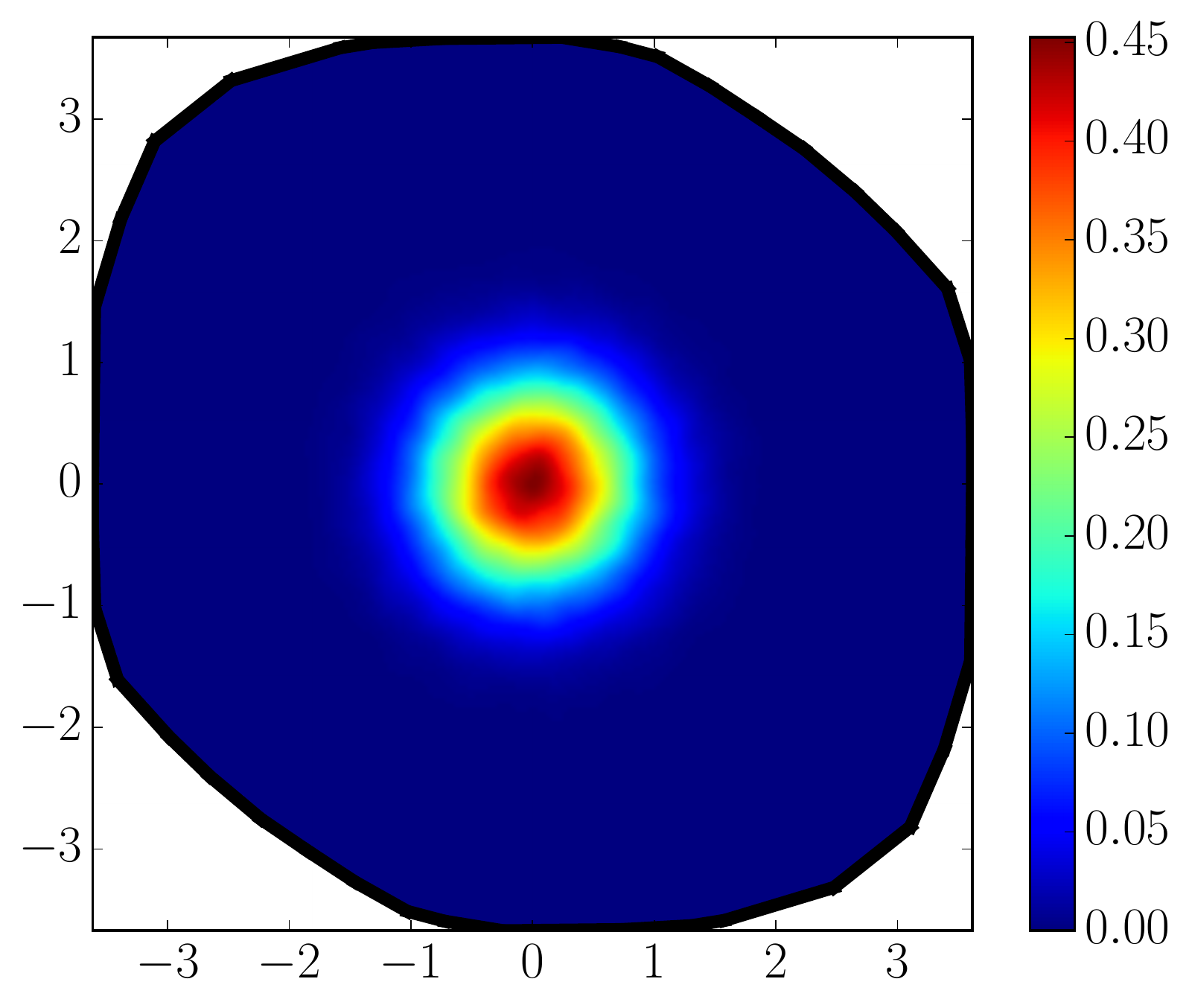}
\includegraphics[width=0.49\textwidth]{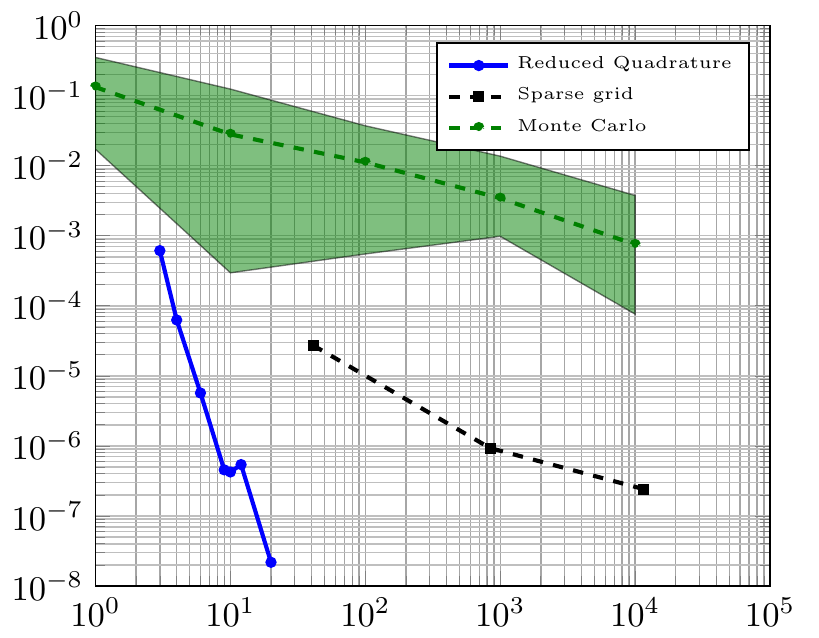}
\end{center}
\caption{(Left) The zonotope defining the domain of integration and the joint probability density of the two variables defining the two-dimensional ridge function of 20 uniform variables. The affine mapping described in Section \ref{sec:non-tp-measures} is used to center the mean of the zonotope density over the highly non-linear region of the chemical reaction model response surface show in the right of Figure~\ref{fig:banana-density}. (Right) The convergence of the error in the mean value of the ridge function computed using reduced quadrature rules over the two-dimensional zonotope compared against more standard quadrature rules in the full 20-dimensional space.}
\label{fig:zonotope}
\end{figure}


\section{Conclusions}\label{sec:conclusion}
In this paper we present a flexible numerical algorithm for generating
polynomial quadrature rules. The algorithm employs optimization to find
a quadrature rule that can exactly integrate, up to a specified
optimization tolerance, a set of polynomial moments.  
We provide novel theoretical
analysis of the minimal number of nodes of a polynomial rule that is exact for a
given set of polynomial moments.  Intuition is developed using a
simple set of analytical multivariate examples that address existence
and uniqueness of optimal rules. In practice we often cannot computationally find an
optimal quadrature rule. However we can find rules that have
significantly fewer points than the number of moments being
matched. Typically the number of points $M$ is only slightly larger ($<10$
points) than the number of moments, $N$, divided by the dimension plus one, $d+1$,
i.e. $M\approx N/(d+1)$.
The algorithm we present is flexible in the sense that it can construct
quadrature rules with positive weights: (i) for any measure for
which moments can be computed; (ii) using analytic or sample based
moments; (iii) for any set of moments, e.g. total-degree,
hyperbolic-cross polynomials. We have shown that this algorithm can be used to efficiently integrate functions in a variety of settings: (i) total-degree integration on hypercubes; (ii) integration for non-tensor-product measures; (iii) high-order integration using approximate moments; (iv) ANOVA approximations; and (v) ridge function integration.


\section{Acknowledgements}
J.D.~Jakeman's work was supported by DARPA EQUIPS. Sandia National
Laboratories is a multimission laboratory managed and operated by
National Technology and Engineering Solutions of Sandia, LLC., a
wholly owned subsidiary of Honeywell International, Inc., for the
U.S. Department of Energy's National Nuclear Security Administration
under contract DE-NA-0003525.

A. Narayan is partially supported by AFOSR FA9550-15-1-0467 and DARPA EQUiPS N660011524053.
\appendix
\section{The LP algorithm}\label{app:lp-algorithm}
In this section we detail the reduced quadrature algorithm first presented in~\cite{ryu_extensions_2014}. 
Let a finite index set $\Lambda$ be given with $|\Lambda| = N$, and suppose that $\left\{ p_j \right\}_{j=1}^N$ is a basis for $P_\Lambda$. In addition, let $r$ be a polynomial on $\R^d$ such that $r \not\in P_\Lambda$. We seek to find a positive Borel measure $\nu$ solving
\begin{align*}
  \textrm{minimize } &\int r(x) \dx{\nu(x)} \\
  \textrm{subject to } &\int p_j(x) \dx{\nu(x)} = \int p_j(x) \dx{\mu(x)}, \hskip 5pt j=1, \ldots, N
\end{align*}
The authors in \cite{ryu_extensions_2014} propose the following procedure for approximating a solution:
\begin{enumerate}[label=\arabic*)]
  \item Choose a candidates mesh, $\left\{y_j\right\}_{j=1}^M$, of $S$ points on $D$. Solve the much more tractable finite-dimensional linear problem
    \begin{align*}
      \textrm{minimize } &\sum_{k=1}^S v_k r(y_k) \\
      \textrm{subject to } &\sum_{k=1}^S v_k p_j(y_k) = \int p_j(x) \dx{\mu(x)}, \hskip 5pt j=1, \ldots, n \\
      \textrm{and } & v_k \geq 0, \hskip 5pt k=1, \ldots, S
    \end{align*}
    The optimization is over the $S$ scalars $v_k$. Let the solution to this problem be denoted $v_k^\ast$.
  \item Identify $M \leq N$ clusters from the solution above. Partition the index set $\left\{1, \ldots, S\right\}$ into these $M$ clusters, denoted $C_j$, $j=1, \ldots, M$. Construct $M$ averaged points and weights from this clustering:
    \begin{align*}
      \widehat{{x}}_j &= \frac{1}{\widehat{w}_j} \sum_{k \in C_j} v_k^\ast y_k, & \widehat{w}_j &= \sum_{k \in C_j} v^\ast_k
    \end{align*}
    Note that the size-$M$ set $\left\{\widehat{{x}}_j, \widehat{w}_j \right\}_{j=1}^M$ is a positive quadrature rule, but it is no longer a solution to the optimization in the previous step.
  \item Solve the nonlinear optimization problem for the nodes ${x}_{(k)}$ and weights ${w}_{(k)}$, $k=1, \ldots, M$,
    \begin{align*}
      \textrm{minimize } &\sum_{j=1}^N \left( \int p_j(x) \dx{\mu(x)} - \sum_{k=1}^M w_k p_j({x}_{(k)}) \right)^2 \\
      \textrm{subject to } & {x}_{(k)} \in \Gamma \textrm{ and } w_k \geq 0,
    \end{align*}
    using the initial guess ${x}_{(k)} \gets \widehat{{x}}_k$ and $w_k \gets \widehat{w}_k$. 
\end{enumerate}
We refer to the above method as the LP algorithm.


\section{Reduced quadrature algorithms}
\label{app:algorithms}
This section presents pseudo code that outlines how to compute reduced quadrature rules as detailed in Section~\ref{sec:algorithm}. Algorithm~\ref{alg:reduced-quadrature} presents the entire set of steps for computing reduced quadrature rules and Algorithm~\ref{alg:cluster} details how the clustering algorithm used to generate the initial condition for the local optimization used to compute the final reduced quadrature rule.
\begin{algorithm}[ht]
\SetKwInOut{Input}{input}\SetKwInOut{Output}{output}
\SetKwRepeat{Do}{do}{while}

\Input{measure $\mu$ with and polynomial family $p$, index set $\Lambda$, quadrature tolerance $\epsilon$}
\Output{quadrature points ${X}=[{x}_{1},\ldots,{x}_{M}]$ and weights $\V{w}=(w_1,\ldots,w_M)^T$}

\BlankLine
With $N = |\Lambda|$, compute moments $\V{m}=(m_1,\ldots,m_{N})^T$, $$m_i=\int_\Gamma,p_\alpha({x})\,d\mu({x}), \quad\forall \alpha\in\Lambda$$

Generate $S$ candidate samples ${X}_S$ over integration domain $D$\;
Compute initial condition: $$\min \,\lVert \V{w}\rVert_1 \,\mathrm{ s.t. }\, \V{\Phi}\V{w}=\V{m}, \quad\Phi_{ij}=p_{\alpha}({x}_{k}),\,k\in[S],j\in[N]$$

Set minimum number of quadrature points $$M=N/(d+1)$$

$i=0$\;
\Do{$\lVert\V{\Phi}\V{w}-\V{m}\rVert_2>\epsilon$}{
Cluster initial condition into $M+i$ points and weights using Algorithm \ref{alg:cluster}\;
Solve \begin{align*}
      \textrm{minimize } &\lVert\V{\Phi}\V{w}-\V{m}\rVert_2 \;\mathrm{\, s.t }\; {x}_{k}\in D , w_k>0
    \end{align*}
Set $i \gets i + 1$
}
\caption{Reduced Quadrature Method}
\label{alg:reduced-quadrature}
\end{algorithm}

\begin{algorithm}[ht]
\SetKwInOut{Input}{input}\SetKwInOut{Output}{output}
\SetKwRepeat{Do}{do}{while}

\Input{Initial quadrature points ${X}=[{x}_{1},\ldots,{x}_{S}]$ and weights $\V{w}=(w_1,\ldots,w_S)^T$, number of clusters $M$ }
\Output{quadrature points $\widehat{{X}}$ and weights $\widehat{\V{w}}$}

\BlankLine

$\widehat{{X}}={X}$ and $\widehat{\V{w}}=\V{w}$\;
\While{$S>M$}{
$I = \argmin_{k\in[N]} \widehat{w}_k$\;
$J = \argmin _{k\in[N]} \lVert \widehat{{x}}_{k}-\widehat{{x}}_{I}\rVert_2$\;
Form new point as convex combination of ${x}_{I}$ and ${x}_{J}$
\begin{align*}
  {x}^\star = \frac{1}{w^\star} (\widehat{w}_I \widehat{{x}}^{I}+ \widehat{w}_J \widehat{{x}}^{J}), \quad\quad w^\star = \widehat{w}_I+\widehat{w}_J
    \end{align*}

Set $\widehat{w}_I={w^\star}$ and $\widehat{{x}}_{I}={x}^\star$\;
Remove $\widehat{w}_J$ and $\widehat{{x}}_{J}$ from $\widehat{X}$ and $\widehat{\V{w}}$\; 
$S=S-1$\;
}
\caption{Cluster Initial Quadrature Rule}
\label{alg:cluster}
\end{algorithm}

\section{Ridge function quadrature}\label{app:ridge-function-moments}
Consider a variable $y\in \mathbb{R}^d$ and a linear transformation $A \in\mathbb{R}^{s\times d}:\R^d\rightarrow \R^s$ which maps the variables $y$ into a lower dimensional set of variables $x = A y \in\mathbb{R}^s$, where $s\le d$. When $\nu$ is a measure in $\R^d$, this transformation induces a measure $\mu$ in $\R^s$. In this section we describe how to compute the moments of the measure $\mu$ in terms of those for $\nu$. Being able to compute such moments allows one to efficiently integrate ridge functions (see Section~\ref{sec:dim-reduction}).

One approach is to approximate the moments of $\mu$ via Monte Carlo sampling. That is, generate a set of samples $Y=\{y_{i}\}_{i=1}^S \subset \R^d$ from the measure $\nu$, and then compute a set of samples $X=AY$ in the $s$-dimension space. Given a multi-index set $\Lambda$ with basis $p_\alpha$, $\alpha \in \Lambda$, the $\mu$-moments of $p_\alpha$ can then be computed approximately via $$\frac{1}{M}\sum_{i=1}^M p_\alpha (x_{i}) = \frac{1}{M} \sum_{i=1}^M p_\alpha( A y_i).$$ 
Such an approach is useful when one cannot directly evaluate 
$\dx{\mu(x)}$ but rather only has samples from the measure. However
the accuracy of a moment matching quadrature rule will be limited by
the accuracy of the moments that are being matched. For moments
evaluated using Monte Carlo sampling the error in these moments decays slowly at a rate proportional to $M^{-1/2}$ and the variance of the polynomial $p_\alpha$.

However, when the higher-dimensional meaure $\nu$ is a tensor-product measure, $\nu(y)=\prod_{i=1}^d\nu_i(y_i)$, with each $\nu_i(x_i)$ a univariate measure, then the moments of the lower-dimensional measure $\mu$ can be computed analytically using, for example, a monomial basis. 

We need to compute the moments of a monomial basis of the variables $x$ with respect to the measure $\mu(x)$. Computing an expression of the measure $\mu(x)$ in terms of $\nu$ is difficult in general. Instead, we leverage the following equality
\begin{align*}
  \int_{D} P_\alpha(y) d\nu(y) &= \int_{D} p_\alpha(x) d\mu(x), & P_\alpha(y) &\coloneqq p_\alpha(A x).
\end{align*}
In particular, monomials in $x$ can be expanded in terms of the variables $y$, e.g $x^p=(A y)^p$, and the resulting polynomials $P_\alpha$ are just products of univariate integrals which can be computed analytically or to machine precision with univariate Gaussian quadrature. 

For example let $y\in[-1,1]^3$,  $\nu(y)$ be the uniform probability measure, and $x=Ay$, where $A\in\mathbb{R}^{2\times 3}$ then the moment of $x_1x_2$ is 
\begin{align*}
  \int_{D} x_1 x_2 d\mu(x)= \int_{[-1,1]^3} (A_{11}y_1+A_{12}y_2+A_{13}y_3)(A_{21}y_1+A_{22}y_2+A_{23}y_3) d\nu(x).
\end{align*}
The right-hand side high-dimensional integrand can be expanded into sums of products of univariate terms, and thus can be integrated with univariate quadrature. 

Let $A$ have rows $a_j^T \in \R^d$, $j = 1, \ldots, s$, and for simplicity assume that $\nu$ is a measure on $[-1,1]^d$. For a general multi-index $\alpha$, we have
\begin{align}\nonumber
  \int_D p_\alpha(x) \dx{\mu}(x) = \int_D x^\alpha \dx{\mu}(x) &= \int_{[-1,1]^d} \prod_{j=1}^s \left( a_j^T y\right)^{\alpha^{(j)}} \dx{\nu}(y) \\\label{eq:multinomial-expansion}
                                                               &= \int_{[-1,1]^d} \prod_{j=1}^s \left[ \sum_{|\beta| = \alpha^{(j)}} \left(\begin{array}{c} \alpha^{(j)} \\ \beta \end{array} \right) a_j^\beta y^\beta \right] \dx{\nu}(y),
\end{align}
where we have, for a generic $\beta \in \N_0^d$, 
\begin{align*}
a_j^T &= \left( \begin{array}{ccc} a_{j,1} & \ldots & a_{j,d} \end{array}\right), & a_j^\beta &= \prod_{k=1}^d a_{j,k}^{\beta^{(k)}},
\end{align*}
and the multinomial coefficients
\begin{align*}
\left(\begin{array}{c} \alpha^{(j)} \\ \beta \end{array} \right) &\coloneqq \frac{\alpha^{(j)} !}{\beta !} = \left( \begin{array}{c} \alpha^{(j)} \\ \beta^{(1)},\,\beta^{(2)},\,\ldots,\,\beta^{(d)}\end{array} \right).
\end{align*}
Inspection of \eqref{eq:multinomial-expansion} and using the tensor-product structure of $\nu$, we see that this can be evaluated exactly via sums and products of univariate integrals of $y^\beta$. While exact, this approach becomes quite expensive for large $k \coloneqq \alpha^{(j)}$ (in which case there are $\left( \begin{array}{c} k + d - 1 \\ k - 1 \end{array}\right)$ summands under the product), or when $s$ in large (in which case one must expand an $s$-fold product). Nevertheless, one can use this approach for relatively large $s$ and $d$ since the univariate integrands in \eqref{eq:multinomial-expansion} are very inexpensive to tabulate, and it is only processing the combination of them that is expensive.

\bibliographystyle{plain}
\bibliography{references}
\end{document}